\documentclass[12pt]{iopart}
\usepackage{iopams}
\usepackage{setstack}
\usepackage{mathrsfs}
\usepackage{cite}
\eqnobysec

\newif\ifaux
\IfFileExists{\jobname.aux}{\auxtrue}{\auxfalse}%
\ifaux\relax\else\renewcommand{\eref}{\relax}\fi

\newtheorem{thm}{Theorem}[section]
\newtheorem{theorem}[thm]{Theorem}

\newtheorem{lemma}[thm]{Lemma}

\newtheorem{proposition}[thm]{Proposition}
\newtheorem{definition}[thm]{Definition}
\newtheorem{remark}[thm]{Remark}

\newenvironment{proof}{\smallskip\noindent{\bf Proof.}\rm}
                        {\hspace*{\fill} $\Box$\medskip}

\newenvironment{proposition*}[1]{\smallskip\noindent{\bf #1.}\it}{\medskip}
\newenvironment{theorem*}[1]{\smallskip\noindent{\bf #1.}\it}{\medskip}
\newenvironment{proofof}[1]{\smallskip\noindent{\bf #1}\rm}
                {\hspace*{\fill} $\Box$\medskip}

\newcommand{\operatorname}{\mathop}

\newcommand\codim{\operatorname{codim}}
\renewcommand\Im{\operatorname{Im}}
\renewcommand\Re{\operatorname{Re}}
\newcommand\Ran{\operatorname{Ran}}
\newcommand\res{\operatorname{res}}

\newcommand\sign{\operatorname {sign}}

\newcommand\rank{\operatorname {rank}}

\newcommand{\bN}{\mathbb N}

\begin{document}

\title[Matrix Sturm--Liouville operators]%
{Inverse spectral problems for Sturm--Liouville operators with matrix-valued potentials}
\author{Ya~V~Mykytyuk and N~S~Trush}%
\address{Lviv National University, 1 Universytets'ka st.,
    79602 Lviv, Ukraine}%
\eads{\mailto{yamykytyuk@yahoo.com}, \mailto{n.s.trush@gmail.com}}

\date{30 June 2009}%

\begin{abstract}
We give a complete description of the set of spectral data (eigenvalues and specially introduced norming constants) for Sturm--Liouville operators on the interval~$[0,1]$ with matrix-valued potentials in the Sobolev space~$W_2^{-1}$ and suggest an algorithm reconstructing the potential from the spectral data that is based on Krein's accelerant method.
\end{abstract}



\section{Introduction}\label{sec.1}

The main aim of this paper is to solve the inverse spectral problem for Sturm--Liouville operators on the unit interval~$[0,1]$ with matrix-valued distributional potentials in the Sobolev space~$W_2^{-1}$. Before explaining the setting, we shall give a short overview of the known results; the reader is referred to the books~\cite{Lev,Mar,PT} and the review paper~\cite{Ges} for further details.

\subsection{Known results}

The study of the inverse spectral problems for Sturm--Liouville and Dirac operators has a rather long history. It was initiated by the celebrated paper by Borg~\cite{BO46} of~1946, who proved that two spectra of a Sturm--Liouville equation determine uniquely the potential. Later Marchenko~\cite{Mar50} showed that the potential is uniquely determined by the spectral function. In 1951, there appeared the paper by Gelfand and Levitan~\cite{GL51} giving an algorithm reconstructing the potential from the spectral measure, and the paper of Marchenko~\cite{Mar51} with a detailed account of his method. Meanwhile Krein developed in a series of papers~\cite{Kr51,Kr53,Kr54,Kr56} an alternative approach to inverse problems.
As is well explained by Ramm~\cite{Ramm}, each of the three methods---that of Gelfand and Levitan, Marchenko, and Krein---enlightens a different facet of the problem and is equally important to the theory.

Another important tool for the study of the inverse spectral problems was developed by Trubowitz and his coauthors Isaacson~\cite{IT83}, Isaacson and  McKean~\cite{IMT84},
Dalhberg~\cite{DT84}, and P\"{o}schel~\cite{PT}. It is based on analytic dependence of an individual eigenvalue and norming constant on the potential. In some cases (e.g. in the inverse problem for the per\-turbed harmonic oscillator) the Trubowitz method is probably the only one yielding a satisfactory result~\cite{CKK04,CK07}.

The study of the matrix-valued Sturm--Liouville and Dirac operators has mainly followed the path of generalizing the methods developed in the scalar case. A solution to the inverse scattering problem for matrix Sturm--Liouville operators on the half-line was given by
 Newton and Jost~\cite{NJ},
 Krein~\cite{Kr55,Kr56},
 Agranovich and Marchenko~\cite{AM63}.
Analogous results for Dirac operators on the half-line were obtained in the papers of Levitan and Gasymov~\cite{LG66} and Gasymov~\cite{Gasym}, and for more general first order systems by Lesch and Malamud~\cite{LM}.

Inverse spectral problems for matrix-valued Sturm--Liouville and Dirac operators on a finite interval have not been discussed so extensively. Uniqueness of reconstruction was studied in the papers by Carlson~\cite{Ca} and Yurko~\cite{Yu1}, and Malamud~\cite{Mal1} investigated uniqueness for various settings of inverse problems for systems of differential equations of the first order. Yurko in~\cite{Yu1} and~\cite{Yu2} suggested an algorithm reconstructing the potential from the spectral data, but did not give a complete description of the set of spectral data. Probably the closest to our work are the papers by Chelkak and Korotyaev~\cite{CK,CK1}, which we  next discuss in more detail.

In~\cite{CK}, the authors investigated uniqueness questions for operators
\begin{equation}\label{eq.1a}
       T_q = -\frac{\rmd^2}{\rmd x^2} + q
\end{equation}
on~$(0,1)$ subject to the Dirichlet boundary conditions and with Hermitian matrix-valued
potentials~$q$ in $L_1(0,1)$ entrywise and studied isospectral transformations as well as some subtle properties of the local structure of the set of isospectral potentials. An important phenomenon was discovered there, namely, the one of the so called ``forbidden'' subspaces. In~\cite{CK1}, the inverse spectral problem for operators $T_q$ with potentials~$q$ belonging to~$L_2(0,1)$ entrywise was treated. The main theorem there gives a complete description of the set of spectral data. That result of~\cite{CK1} and Theorem~\ref{th.15} of the present paper are similar but they treat different classes of potentials. Also the approaches are different: while the proofs in~\cite{CK1} are based on the Trubowitz method and some special isospectral transformations, we use the Krein accelerant method.


\subsection{Setting of the problem}
Let $M_r$ be the Banach algebra of the square~$r\times r$ matrices with complex entries, which we identify with the Banach algebra of linear operators~$\mathbb
C^r\to \mathbb C^r$ endowed with the standard norm (see~\ref{add.1}).
We shall write $I$ for the unit element of~$M_r$ and
$M^+_r$ for the set of all $A\in M_r$ such that $A=A^*\ge 0$.
We shall use the abbreviations
\begin{equation*}
        W_p^s:=W_p^s((0,1),\mathbb{C}^r),\qquad \boldsymbol
        W_p^s:=W_p^s((0,1),\mathit{M}_r) \qquad(s\in\mathbb{Z}, p\geq 1)
\end{equation*}
for the corresponding Sobolev spaces, as well as the notations
\begin{eqnarray*}
& L_p:=L_p((0,1), \mathbb{C}^r),\qquad &\mathbf L_p:=L_p((0,1),\mathit{M}_r), \qquad p\geq1,\\
& \Re \mathbf L_p:=\{u\in \mathbf L_p|u=u^{\ast}\},\qquad &
\Re\mathbf W^{-1}_2:=\{u\in \mathbf W^{-1}_2|u=u^{\ast}\}.
\end{eqnarray*}
Some further information on these spaces can be found in~\ref{add.1}.

For an arbitrary $\tau\in\mathbf L_2$, we consider the
differential expression
\begin{equation*}
         {\mathfrak t_\tau}(f):= -\left(\frac{\rmd}{\rmd x}+\tau\right)\left(\frac{\rmd}{\rmd x}-\tau \right)f
\end{equation*}
on the domain
\begin{equation*}
         \mathit{D}(\mathfrak t_\tau):=\{f\in W_2^1 \mid
         f_\tau^{[1]}:= (f'-\tau f)\in W_2^1\}.
\end{equation*}
The function $f_\tau^{[1]}$ is usually called the \emph{quasi-derivative} of~$f$.

Denote by $\mathfrak t_{\tau,D} $ and $\mathfrak t_{\tau,N} $ the restrictions of
$\mathfrak t_{\tau} $ onto the domains
\begin{eqnarray*}\label{eq.12i}
             \mathit{D}(\mathfrak t_{\tau,D})&:=\{f\in \mathit{D}(\mathfrak t_\tau) \mid
                                                  f\in W_{2,0}^1 \},\\
             \mathit{D}(\mathfrak t_{\tau,N})&:=\{f\in \mathit{D}(\mathfrak t_\tau) \mid
                                                  f_{\tau}^{[1]}\in W_{2,0}^1 \}
\end{eqnarray*}
respectively. Here and hereafter, $W_{2,0}^1:=\{f\in W_2^{1}\,|\,f(0)=f(1)=0\}.$

The differential expression $\mathfrak t_{\tau,D} $ (unlike $\mathfrak
t_{\tau,N} $) can be written in the usual potential form. Namely, let us define the Miura map~$b$ via (cf.~\cite{KPST})
\begin{equation*}
            \mathbf L_2 \ni u\mapsto u'+u^2=:b(u)\in \mathbf W_2^{-1}
\end{equation*}
and the class of (Hermitian) Miura potentials
\begin{equation*}
           \mathfrak M := \{q=b(u)\mid u\in \mathbf L_2\}, \qquad
           \left(\Re\mathfrak M := \{q=b(u)\mid u\in \Re\mathbf L_2\}\right).
\end{equation*}

It can be shown (see Lemma~\ref{le.22}) that for an arbitrary
$q\in \mathfrak M$ and $\tau\in b^{-1}(q)$ one has
\begin{equation*}
         \mathit{D}(\mathfrak t_{\tau,D})=\{f\in  W_{2,0}^1 \mid (-f''+q f)\in L_2 \},\qquad
         \mathfrak t_{\tau,D}(f)=-f''+qf,
\end{equation*}
where the derivative~$f''$ and the product $qf$ should be
understood in the distributional sense. In particular,
\begin{equation}\label{eq.11a}
         \mathfrak t_{\tau_1,D}=\mathfrak t_{\tau_2,D}\quad \mathrm{if}\quad
          b(\tau_1)=b(\tau_2).
\end{equation}

For $\tau\in\mathbf L_2$ and $q:=b(\tau)$, we consider the
operators $S_\tau$ and $T_q$ acting in~$L_2$ via
\begin{eqnarray*}
                    S_\tau f &:=\frak t_{-\tau,N}(f), \qquad \qquad
                               \mathit{D}(S_\tau) &:=\mathit{D}(\frak t_{-\tau,N}),\\
                    T_qf &:=-f''+q f,\qquad \quad \,\,
                               \mathit{D}(T_q) &:=\mathit{D}(\frak t_{\tau,D}).
\end{eqnarray*}
In this paper we shall concentrate ourselves on the study of the spectral properties of the operators $S_\tau$ and $T_q$ for
 $\tau\in \Re\mathbf L_2$ and $q= b(\tau)$. In this case the operators $S_\tau$ and $T_q$ are self-adjoint; moreover, $S_\tau\ge 0$ and
$T_q>0$. Their spectra $\sigma(S_\tau)$ and $\sigma(T_q)$ consist of a countably many
isolated eigenvalues accumulating at~$+\infty$; moreover,
$\sigma(S_{\tau})=\sigma(T_q)\cup\{0\}$.

Let $\tau\in\mathbf L_2$ and $\lambda\in\mathbb C$. We denote by
$\varphi(\cdot,\lambda,\tau)$ and $\psi(\cdot ,\lambda,\tau)$ the
matrix-valued solutions of the Cauchy problems
\begin{eqnarray}\label{eq.11}
       -\left(\frac{\rmd}{\rmd x}+\tau\right)\left(\frac{\rmd}{\rmd x}-\tau \right)\varphi=\lambda^2\varphi,\qquad
                 \varphi(0)=0,\qquad
                 \varphi_{\tau}^{[1]}(0)=\lambda I,\\
                    \label{eq.12}
      -\left(\frac{\rmd}{\rmd x}-\tau\right)\left(\frac{\rmd}{\rmd x}+\tau \right)\psi=\lambda^2 \psi,\qquad
                 \psi(0)=I,\qquad
                 \psi_{-\tau}^{[1]}(0)=0.
\end{eqnarray}
They are related to each other via
\begin{equation*}
            \left(\frac{\rmd}{\rmd x}-\tau\right)\varphi(\cdot,\lambda,\tau)
                                      =\lambda\psi(\cdot,\lambda,\tau), \qquad
            \left(\frac{\rmd}{\rmd x}+\tau\right)\psi(\cdot,\lambda,\tau)
                                      =-\lambda\varphi(\cdot,\lambda,\tau)
\end{equation*}
and admit representations in the form (see Theorem~\ref{th.21})
\begin{equation}\label{eq.13}
        \eqalign{
                  \varphi(x,\lambda,\tau)=\sin \lambda x I+\int_{0}^{x}(\sin\lambda t) K_{\tau,D}(x,t)\,\rmd t, \\
        \eqalign
                  \psi(x,\lambda,\tau)=\cos \lambda x I+\int_0^x (\cos\lambda t)K_{\tau,N}(x,t)\,\rmd t,}
\end{equation}
where the matrix-valued kernels $K_{\tau,D}$ and $K_{\tau,N}$
belong to the algebra~$G_2^+$ (see~\ref{add.1}).
Equalities~\eref{eq.13} determine uniquely $K_{\tau,D}$ and
$K_{\tau,N}$  within the class $L_2(\Omega, M_r)$,
$\Omega:=\{(x,t)\mid 0\le t\le x\le 1 \}.$

The function $\lambda\mapsto \varphi(1,\lambda,\tau)^{-1}$, as well as the Weyl--Titchmarsh function
\begin{equation*}\label{eq.14}
         m_\tau(\lambda):=-\varphi(1,\lambda,\tau)^{-1}\psi(1,\lambda,-\tau)
\end{equation*}
are meromorphic in~$\mathbb C$; notice that
$m_0(\lambda)=-\cot\lambda I$.

Let now $\tau\in\Re\mathbf L_2$ and $q=b(\tau)$. We denote by
$\lambda_j(\tau)$ $(j\in \mathbb{Z_+})$ the square roots of the
pairwise distinct eigenvalues of the operator $S_\tau$ labelled in
increasing order, i.e.,
$0=\lambda_0(\tau)<\lambda_{k}(\tau)<\lambda_{k+1}(\tau)$,
$k\in\mathbb N$; then
\begin{equation*}
         \sigma(S_\tau)=\{\lambda_j^2(\tau)\}_{j=0}^{\infty},\qquad
         \sigma(T_q)=\{\lambda_j^2(\tau)\}_{j=1}^{\infty}.
\end{equation*}
The function $m_\tau$ is a
matrix-valued Herglotz function (i.e., $\Im m_\tau(\lambda)\ge 0$
for $\Im\lambda>0$), and the set
$\{\pm\lambda_j(\tau)\}_{j\in\mathbb Z_+}$ is the set of its
poles. Put by definition
\begin{equation}\label{eq.15}
         \alpha_0(\tau):=-\frac12\underset{\lambda=0}{\res} \,\,m_\tau(\lambda), \qquad
         \alpha_j(\tau):=-\underset{\lambda=\lambda_j}{\res} \,\,m_\tau(\lambda), \qquad
                       j\in \mathbb N.
\end{equation}
The matrix $\alpha_j(\tau)$ for $j\in\mathbb Z_+ $ (resp.\, for $j\in\mathbb
N$) is called the \emph{norming constant} of the operator~$S_\tau$
(resp.\, of the operator~$T_q$) corresponding to the eigenvalue~$\lambda^2_j(\tau)$. We note that the multiplicity of the eigenvalue~$\lambda_j^2(\tau)$ of $S_\tau$ or $T_q$ equals~$\rank \alpha_j(\tau)$.

It turns out that a sequence $((\lambda_j(\tau),\alpha_j(\tau)))_{j\in \mathbb{N}}$ depends only on the function $q=b(\tau)$. In view of
this we call the collections $\mathfrak a_\tau=((\lambda_j(\tau),\alpha_j(\tau)))_{j\in \mathbb{Z_+}}$ and
$\mathfrak b_q =((\lambda_j(\tau),\alpha_j(\tau)))_{j\in
\mathbb{N}}$ a sequences of \emph{spectral data}, and the matrix-valued
measures
\begin{equation}\label{eq.16}
         \nu_\tau:=\sum_{j=0}^{\infty}\alpha_j(\tau)\delta_{\lambda_j(\tau)}, \qquad
         \mu_q:=\sum_{j=1}^{\infty}\alpha_j(\tau)\delta_{\lambda_j(\tau)}
\end{equation}
will be termed the \emph{spectral measures} of the operators $S_\tau$ and
$T_q$ respectively. Here $\delta_{\lambda}$ is the Dirac
delta-measure centred at the point $\lambda$. In particular, if
$\tau=q=0$, then
\begin{equation}\label{eq.17}
         \nu_0=\frac12 I\delta_0+\sum\limits_{n=1}^{\infty}I\delta_{\pi n}, \qquad
         \mu_0=\sum\limits_{n=1}^{\infty}I\delta_{\pi n}.
\end{equation}
Clearly, the spectral measures of the operators~$S_\tau$ and $T_q$ are related by the simple formula
\begin{equation*}
         \nu_{\tau}-\mu_q=\alpha_0(\tau)\delta_0, \qquad
         \tau\in\Re\mathbf L_2, \quad
                   q=b(\tau).
\end{equation*}
A simple relation exists also between the function~$m_\tau$ and the measure $\nu_\tau$. Namely,
\begin{equation*}
          m_\tau(\lambda)=2\lambda\int_0^{\infty}\frac{\rmd \nu_\tau(\xi)}{\xi^2-\lambda^2}, \qquad
          \lambda\in\mathbb C.
\end{equation*}

The main aim of the present paper is to give a complete description of the classes
    $\mathfrak A :=\{\mathfrak a_{\tau}\mid \tau\in \Re \mathbf L_2\}$
and
    $\mathfrak B:=\{\mathfrak b_q\mid q\in \Re\mathfrak M \}$ of spectral data
and to suggest an efficient method of reconstructing the
functions~$\tau$ and $q$ from the measures~$\nu_{\tau}$ and~$\mu_q$ respectively. We note that the description of the
classes $\mathfrak A$ and $\mathfrak B$ is equivalent to the description of
the families of measures
  $\mathcal V:=\{\nu_{\tau}\mid\tau\in\Re\mathbf L_2\}$ and $\mathcal M:=\{\mu_q\mid q\in \Re\mathfrak M \}$.

\subsection{Main results}\label{subsec.13}

We start with characterization of the spectral data for Sturm--Liouville operators under consideration. In what follows $\mathfrak a$ (resp.~$\mathfrak b$) will stand for an arbitrary sequence
$((\lambda_j,\alpha_j))_{j\in\mathbb Z_+}$
(resp.~$((\lambda_j,\alpha_j))_{j\in\mathbb N}$), in which
$(\lambda_{j})_{j\in\mathbb Z_+}$ is a strictly increasing sequence of non-negative numbers and  $\alpha_j$ are nonzero matrices in $M^+_r$, and $\nu^\mathfrak a$ and $\mu^\mathfrak
b$ will denote the measures given by
\begin{equation}\label{eq.111}
         \nu^\mathfrak a:=\sum_{j=0}^{\infty}\alpha_j\delta_{\lambda_j}, \qquad
          \mu^\mathfrak b:=\sum_{j=1}^{\infty}\alpha_j\delta_{\lambda_j}.
\end{equation}
Next, we partition the semi-axis $[0,\infty)$ into pairwise
disjoint intervals $\Delta_n$ $(n\in\mathbb Z_+)$, viz.
\begin{equation*}
         \Delta_0=\{0\},\qquad \Delta_1=\Bigl(0,{3\pi\over2}\Bigr], \qquad
             \Delta_n=\Bigl(\pi n-\frac{\pi}2,\pi n +\frac{\pi}2\Bigr], \quad
                     n>1.
\end{equation*}
A complete description of the classes $\mathfrak A$ and $\mathfrak B$ is given by the following two theorems.

\begin{theorem}\label{th.14}
In order that a sequence $\mathfrak
a=((\lambda_j,\alpha_j))_{j\in\mathbb Z_+}$ should belong to
$\mathfrak A$ it is necessary and sufficient that the following
conditions be satisfied:
\begin{itemize}
        \item[$(A_1)$]  $\sum\limits_{n=1}^{\infty}\sum\limits_{\lambda_j\in\Delta_n}|\lambda_j-\pi n|^2<\infty, \quad
                              \sup\limits_{n\in\mathbb N}\sum\limits_{\lambda_j\in\Delta_n} 1<\infty, \quad
                         \sum\limits_{n=1}^{\infty}\|I-\sum\limits_{\lambda_j\in\Delta_n}\alpha_j\|^2<\infty;$
        \item[$(A_2)$]  $\exists N_0\in\mathbb N\quad \forall N\in\mathbb N \qquad (N\ge N_0) \Longrightarrow
                              \sum\limits_{n=1}^{N}\sum\limits_{\lambda_j\in\Delta_n}\rank \alpha_j=Nr;$
        \item[$(A_3)$]  the system of functions
                                      $\{d\cos \lambda_{j} x\mid j\in\mathbb{Z_+}, d\in\Ran \alpha_j\}$
                        is complete in the space~$L_2$.
\end{itemize}
\end{theorem}

\begin{theorem}\label{th.15}
In order that a sequence $\mathfrak
b=((\lambda_j,\alpha_j))_{j\in\mathbb N}$ should belong to the
class~$\mathfrak B$ it is necessary and sufficient that the
conditions $(A_1)$, $(A_2)$, and the following condition~$(A_4)$
be satisfied:
\begin{itemize}
        \item[$(A_4)$] the system of functions
                             $\{d\sin \lambda_j x\mid j\in\mathbb{N}, d\in\Ran\alpha_j\}$
                       is complete in the space~$L_2$.
\end{itemize}
\end{theorem}

There exists a simple relation between the classes~$\mathfrak A$ and~$\mathfrak B$.

\begin{proposition}\label{pr.16}
Let $\mathfrak a=((\lambda_j,\alpha_j))_{j\in\mathbb Z_+}$ and
$\mathfrak b=((\lambda_j,\alpha_j))_{j\in\mathbb N}$. Then
  \begin{equation*}
         (\mathfrak a\in\mathfrak A) \Longleftrightarrow (\mathfrak b\in\mathfrak B) \wedge (\lambda_0=0, \alpha_0>0).
  \end{equation*}
\end{proposition}

By definition, every $\mathfrak a \in \mathfrak A$ (resp., every $\mathfrak b \in \mathfrak B$) forms spectral data for an ope\-ra\-tor~$S_\tau$ with some $\tau \in \Re {\mathbf L}_2$ (resp., for an operator~$T_q$ with some $q\in \Re\mathfrak M$). It turns out that these spectral data determine the matrix-valued functions $\tau$ and $q$ uniquely, i.e., the following holds true.

\begin{theorem}\label{th.1uniq}
The mappings
  $\Re {\mathbf L}_2 \ni\tau \mapsto {\mathfrak a}_\tau \in {\mathfrak A}$ and
  $\Re {\mathfrak M} \ni q \mapsto {\mathfrak b}_q \in {\mathfrak B}$
are bijective.
\end{theorem}

As we mentioned earlier, we base our algorithm of reconstruction of the functions~$\tau$ and $q$ on the Krein accelerant method developed in~\cite{Kr}.

\begin{definition}\label{de.11}
We say that a function~$H\in L_2((-1,1),M_r)$ is the \emph{accelerant} if
it is even $($i.e., $H(-x)=H(x))$ and if for every $a\in [0,1]$
the integral equation
\begin{equation*}
               f(x)+\int_0^a H(x-t) f(t)\,\rmd t =0, \qquad  x\in (0,1),
\end{equation*}
has no non-trivial solutions in the space $L_2$. The set of all
accelerants is denoted by~$\mathfrak H_2$ and is endowed
with the metric of the space~$L_2((-1,1),M_r)$. We shall write~$\Re\mathfrak H_2$ for the subset of~$\mathfrak H_2$ of all Hermitian accelerants.
\end{definition}

Spectral measure of the operator~$S_\tau$ naturally generate Krein  accelerants, as explained in the following theorem.

\begin{theorem}\label{th.41}
Take a sequence $\mathfrak a=((\lambda_j,\alpha_j))_{j\in\mathbb
Z_+}$ satisfying condition $(A_1)$ and set $\nu:=\nu^\mathfrak a$. Then the
limit
\begin{equation}\label{eq.18}
         H_\nu(x):=\lim\limits_{n\to\infty}\int_0^{\pi(n+1/2)}2\cos(2\lambda x)\,\rmd (\nu-\nu_0)(\lambda), \qquad
                          x\in (-1,1),
\end{equation}
exists in the topology of the space $L_2((-1,1),M_r)$. If, in addition, condition $(A_3)$ holds, then the function $H_\nu$ belongs to~$\Re\mathfrak H_2$.
\end{theorem}

Conversely, every accelerant~$H$ determines a function in~${\mathbf L}_2$ in the following way. It is known (see~\ref{add.1}) that the associated Krein equation
\begin{equation}\label{eq.19}
         R(x,t)+H(x-t)+\int_0^x R(x,\xi)H(\xi-t)\,\rmd \xi=0, \qquad
         (x,t)\in\Omega,
\end{equation}
has a unique solution~$R_H$ in the class~$L_2(\Omega,M_r)$. We
can now define a mapping $\Theta$ from~$\mathfrak H_2$ to~$
\mathbf L_2$ given by the formula
\begin{equation}\label{eq.110}
         [\Theta(H)](x):=H(x)+\int_0^xR_H(x,\xi)H(\xi)\,\rmd \xi=-R_H(x,0).
\end{equation}

The functions~$\Theta(H)$ and $R_H$ are related to each other as follows:

\begin{theorem}\label{th.12}
Assume that $H\in\mathfrak H_2$, $R=R_H$, and $\tau=\Theta(H)$.
Then
\begin{eqnarray*}
         K_{\tau,D}(x,t)=\frac 12\left[R\Bigl(x,\frac{x+t}{2}\Bigr)-R\Bigl(x,\frac{x-t}{2}\Bigr)\right],\\
         K_{\tau,N}(x,t)=\frac{1}{2}\left[R\Bigl(x,\frac{x+t}{2}\Bigr)+R\Bigl(x,\frac{x-t}{2}\Bigr)\right],
\end{eqnarray*}
where~$(x,t)\in\Omega$ and $K_{\tau,D}$ and $K_{\tau,N}$ are the
kernels of~\eref{eq.13}.
\end{theorem}

Theorem~\ref{th.12} is an analogue of a theorem from the paper of Krein~\cite{Kr}, where the accelerant theory was applied to the inverse scattering theory. Accelerants were used to solve  the inverse spectral problems on finite intervals in the papers~\cite{AHrMy1} and~\cite{AHrMy2}.

The following theorem describes some additional properties of the mapping~$\Theta$.

\begin{theorem}\label{th.13}
The mapping~$\Theta$ is a homeomorphism between the metric
spaces~$\mathfrak H_2$ and~$\mathbf L_2$. Moreover, if
$H\in\mathfrak H_2$, then $H^*\in\mathfrak H_2$ and
$\Theta(H^*)=[\Theta(H)]^*$.
\end{theorem}

Finally, we show how the accelerants can be used to reconstruct $\tau$ and $q$ from the corresponding spectral data.
\begin{theorem}\label{th.17}
 \begin{enumerate}
     \item Assume that $\tau\in\Re{\mathbf L}_2$ and that~${\mathfrak a}:= {\mathfrak a}_\tau$ is the spectral data for the operator~$S_\tau$. Then $\tau = \Theta(H_\nu)$ for $\nu = \nu^{\mathfrak a}$.
     \item Assume that $q\in\Re\mathfrak M$ and that~${\mathfrak b}:= {\mathfrak b}_q$ is the spectral data for the operator~$T_q$. Set $\nu:=I\delta_0 + \mu^{\mathfrak b}$ and $\tau:= \Theta(H_\nu)$; then $q = b(\tau)$.
 \end{enumerate}
\end{theorem}

According to Theorem~\ref{th.17}, the reconstruction algorithm can proceed as follows. Given $\mathfrak a \in \mathfrak A$, we construct a measure~$\nu=\nu^\mathfrak a$ via~\eref{eq.111}, which defines an accelerant~$H=H_\nu$ by formula~\eref{eq.18}. Solving the Krein equation~\eref{eq.19}, we find the kernel~$R_H$, which gives $\tau:=\Theta(H)$ via~\eref{eq.110}. That $\tau$ so constructed is the matrix-function looked for follows from the fact that the matrix Sturm--Liouville operator~$S_\tau$ has the spectral data~$\mathfrak a$ we have started with. As in~\cite{Ramm}, we visualize the reconstruction algorithm by means of the diagram
\begin{equation*}\label{eq.120}
          \mathfrak a \overset{\eref{eq.111}}{\underset{s_1}\longrightarrow}
                  \nu^\mathfrak a=\nu \overset{\eref{eq.18}}{\underset{s_2}\longrightarrow}
           H_\nu=H\overset{\eref{eq.19}}{\underset{s_3}\longrightarrow}
                  R_H \overset{\eref{eq.110}}{\underset{s_4}\longrightarrow} \Theta(H)=\tau.
\end{equation*}
In this diagram, $s_j$ denotes the step number~$j$. Steps $s_1$,
$s_2$, and $s_4$ are trivial. The basic (and non-trivial) step is
$s_3$.

The procedure of reconstructing~$q$ looks similarly. Namely, given $\mathfrak b\in \mathfrak B$, the corresponding potential $q\in\Re\mathfrak M$ can be found following the steps in the next diagram:
\begin{equation*}
          \fl    \mathfrak b \overset{\eref{eq.111}}{\underset{s_1}\longrightarrow}
                      \mu^\mathfrak b {\underset{s_2}\longrightarrow}
                 I\delta_0 +\mu^\mathfrak b=\nu \overset{\eref{eq.18}}{\underset{s_3}\longrightarrow}
                       H_\nu=H \overset{\eref{eq.19}}{\underset{s_4}\longrightarrow}
                  R_H \overset{\eref{eq.110}}{\underset{s_5}\longrightarrow}
                       \Theta(H)=\tau {\underset{s_6}\longrightarrow}
                  b(\tau)=q.
\end{equation*}
We observe that the spectral data for the operator~$T_q$ with $q=b(\tau)$ do not determine~$\tau$; in fact, the Riccati equation $b(\tau)=q$ has many solutions, and by~\eref{eq.11a} each such solution generates the same operator~$T_q$. Also, the assumption that $q\in \Re{\mathfrak M}$ is not restrictive and was made only to simplify the presentation. Indeed, for every $q\in \Re {\mathbf W}_{2}^{-1}$ the corresponding Dirichlet Sturm--Liouville operator~$T_q$ of~\eref{eq.1a} is bounded below and thus becomes positive after adding a suitable multiple $cI$ of the unit matrix~$I$; it then follows that $q + cI$ belongs to~$\Re{\mathfrak M}$ (cf.~\cite{KPST}).

\subsection{The structure of the paper} The paper is organised as follows. Section~\ref{sec.2} deals with the direct spectral problem and consists of four parts. We establish basic properties of the operators $T_q$ and $S_\tau$ (Theorem~\ref{th.24}) in Subsection~\ref{subsec.21}, find the asymptotics of eigenvalues and norming constants
(Theorem~\ref{th.25}) in Subsection~\ref{subsec.22}, study properties of the Weyl--Titchmarsh function in Subsection~\ref{subsec.23}, and, finally, prove necessity of conditions of Theorems~\ref{th.14} and~\ref{th.15} in Subsection~\ref{subsec.24}.

In Section~\ref{sec.3} we study Krein accelerants. This section is divided in three parts. In the first two subsections we establish Theorems~\ref{th.12} and~\ref{th.13}, and in Subsection~\ref{subsec.33} we prove Propositions~\ref{pr.37} and~\ref{pr.38} on operators $\mathscr H_o$ and $\mathscr H_e$, which play an important role in Section~\ref{sec.4}.

The inverse spectral problem is considered in Section~\ref{sec.4}, which consists of three parts. In Subsection~\ref{subsec.41} we study properties of the function~$H_\nu$ (Theorem~\ref{th.41}), in Subsection~\ref{subsec.42} we finish the proof of Theorems~\ref{th.14} and~\ref{th.15}, and, finally, in Subsection~\ref{subsec.43} we establish Proposition~\ref{pr.16} and Theorems~\ref{th.1uniq} and~\ref{th.17}.

There are two short appendices. Some information on the spaces used in the paper and well-known facts from the theory of factorization of Fredholm operators are gathered in~\ref{add.1}. Three auxiliary lemmata on orthogonal projectors are proved in~\ref{add.2}.

\section{{The direct spectral problem}\label{sec.2}}

\subsection{{Basic properties of the operators $T_q$ and $S_\tau$}\label{subsec.21}}
 In this subsection we prove self--adjointness of
the operators $T_q$ and $S_\tau$ in the case where $\tau\in
\Re\mathbf L_2$ and $q\in \mathfrak M$ and also construct their
resolvents and the resolutions of identity.

Recall that we have denoted by $\varphi(\cdot,\lambda,\tau)$ and
$\psi(\cdot,\lambda,\tau)$  the matrix--valued solutions of the
Cauchy problems~\eref{eq.11} and~\eref{eq.12}. The results
of~\cite{MTr} imply the following statement.

\begin{theorem}\label{th.21}
The problems~\eref{eq.11} and~\eref{eq.12} have unique solutions
$\varphi(\cdot,\lambda,\tau)$ and $\psi(\cdot,\lambda,\tau)$ in
the
classes~$\{u\in \mathbf W_2^{1}\,|\,u_{\tau}^{[1]}\in \mathbf W_2^{1}\}$
    and $\{u\in \mathbf W_2^{1}\,|\,u_{-\tau}^{[1]}\in \mathbf W_2^{1}\}$,
respectively. Moreover, the following statements hold:
\begin{enumerate}
\item  the functions $\varphi(\cdot,\lambda,\tau)$ and
$\psi(\cdot,\lambda,\tau)$  are related to each other via
              \begin{equation}\label{eq.21}\fl
                  \left(\frac{\rmd}{\rmd x}-\tau\right)\varphi(\cdot,\lambda,\tau)=\lambda\psi(\cdot,\lambda,\tau),\qquad
                  \left(\frac{\rmd}{\rmd x}+\tau\right)\psi(\cdot,\lambda,\tau)=-\lambda \varphi(\cdot,\lambda,\tau);
               \end{equation}
\item  for every $\tau\in \mathbf L_2$ there exist unique matrix-valued functions
$K_{\tau,D}$ and $K_{\tau,N}$ belonging to the algebra~$G_2^+$
(see~\ref{add.1}) such that for any $\lambda\in \mathbb C$ and
$x\in [0,1]$
              \begin{equation}\label{eq.22}
                 \eqalign{ \varphi(x,\lambda,\tau)=\sin \lambda x I+\int_0^x(\sin\lambda t)K_{\tau,D}(x,t)\,\rmd t, \\
                 \eqalign   \psi(x,\lambda,\tau)=\cos \lambda x I+\int_0^x(\cos\lambda t)K_{\tau,N}(x,t)\,\rmd t; }
              \end{equation}
\item  the mappings $\mathbf L_2\ni\tau\mapsto K_{\tau,D}\in
G_2^+$ and $\mathbf L_2\ni \tau\to K_{\tau,N}\in G_2^+$ are
continuous.
\end{enumerate}
\end{theorem}

\begin{lemma}\label{le.22}
 Suppose that $\tau\in \mathbf L _2$ and $q=b(\tau)$. Then
\begin{equation*}
          \mathit{D}(\mathfrak t_{\tau,D})=\{f\in  W_{2,0}^1\mid(-f''+q f)\in L_2 \},\qquad
          \mathfrak t_{\tau,D}(f)=-f''+qf,
\end{equation*}
where the derivative $f''$ and the product $qf$ are interpreted in
the distributional sense, i.e., as elements of the space
$W_2^{-1}$. Moreover, if $\tau_1$ is another element of~$\mathbf L
_2$, then
\begin{enumerate}
          \item $\mathfrak t_{\tau_1,D}=\mathfrak t_{\tau,D} \Longleftrightarrow b(\tau_1)=b(\tau);$
          \item $K_{\tau_1,D}=K_{\tau,D}\Longleftrightarrow b(\tau_1)=b(\tau);$
          \item $K_{\tau_1,N}=K_{\tau,N}\Longleftrightarrow \tau_1=\tau$.
\end{enumerate}
\end{lemma}

\begin{proof}
If $u\in \mathbf W_2^{-1}$ and $f\in W_{2,0}^1$, then we define
the product $u f\in W_2^{-1}$  as
$$
  (u f)_j=\sum\limits_{k=1}^r u_{jk}f_k, \qquad f=(f_1,\dots,f_r),\quad
                                                u=(u_{jk})_{1\le j,k\le r}.
$$
Observe that for every $f\in W_{2,0}^1$ the equality~$(\tau f)'=\tau' f +\tau f'$ holds in the distributional sense. If, therefore, $f\in  \mathit{D}(\mathfrak t_{\tau,D})$, then
$$
  \mathfrak t_{\tau,D}(f)= -f''+(\tau f)'-\tau f'+\tau^2 f=-f''+qf \in L_2.
$$
On the other hand, if $f\in W_{2,0}^1$ and $(-f''+qf)\in L_2$,
then taking into account the equality
\begin{equation}\label{eq.23}
           -(f'-\tau f)'=(-f''+qf)+\tau(f'-\tau f),
\end{equation}
we obtain $(f'-\tau f)\in W_1^1$. Moreover, the function $f'-\tau
f$ is bounded, and, using~\eref{eq.23} again, we obtain that
$(f'-\tau f)\in W_2^1$, i.e., that $f\in
 \mathit{D}(\mathfrak t_{\tau,D})$.
This establishes the claim about~$\mathfrak t_{\tau,D}$.

Part \emph{(i)} is the straightforward consequence of the above
considerations. Let us prove part~\emph{(ii)}. Assume that
$K_{\tau_1,D}=K_{\tau,D}$. Then $\varphi(\cdot,0,\tau_1)=\varphi(\cdot,0,\tau)$ in view of~\eref{eq.22}, and~\eref{eq.21} yields the relations
\begin{eqnarray}\label{eq.24}
         (\tau_1-\tau)\varphi(\cdot,\lambda,\tau)
               =\lambda[\psi(\cdot,\lambda,\tau)-\psi(\cdot,\lambda,\tau_1)],\\
    \label{eq.25}
           \left(\frac{\rmd}{\rmd x}+\tau_1\right)\left(\frac{\rmd}{\rmd x}-\tau_1\right)\varphi(\cdot,\lambda,\tau)
          =\left(\frac{\rmd}{\rmd x}+\tau\right)\left(\frac{\rmd}{\rmd x}-\tau\right)\varphi(\cdot,\lambda,\tau).
\end{eqnarray}
Note that according to the Riemann--Lebesgue lemma for every
$x_0\in (0,1]$ there exists $\lambda_0\in \mathbb R$ such that
$\det\varphi(x_0,\lambda_0,\tau)\ne 0.$  Therefore we conclude
from the equation~\eref{eq.24} that the function $g=\tau_1-\tau$
is absolutely continuous on the interval $(0,1)$. Putting
$\tau_1=\tau+g$ in~\eref{eq.25}, we easily obtain the equality
$$
   (g'+g\tau+\tau g + g^2)\varphi(\cdot,\lambda,\tau)=0, \qquad \lambda\in\mathbb C,
$$
which implies that $g' + g\tau + \tau g + g^2=0$ and, thus, that
$b(\tau_1)=b(\tau)$.

Now let $b(\tau_1)=b(\tau)$. Then the function $g=\tau_1-\tau$
obeys $g'=\tau^2 -\tau^2_1\in \mathbf L_1$ and thus is absolutely
continuous on the interval~$[0,1]$. Next, the relation $g'+g\tau +
\tau g + g^2=0$ yields equality~\eref{eq.25}. It follows
from~\eref{eq.25} and the definition of the functions
$\varphi(\cdot,\,\lambda,\tau_1)$ and
$\varphi(\cdot,\lambda,\tau)$ that  they are solutions of the
Cauchy problem
\[
\left(\frac{\rmd}{\rmd x}+\tau_1\right)\left(\frac{\rmd}{\rmd
x}-\tau_1\right) y= -\lambda^2 y ,\qquad y(0)=0,\quad
y^{[1]}_{\tau_1}(0)=\lambda I.
\]
Thus $\varphi(\cdot,\lambda,\tau_1)=\varphi(\cdot,\lambda,\tau)$
for all $\lambda\in\mathbb C$. Therefore (see~\eref{eq.22}),
\[
    \int_0^x(\sin\lambda t)[K_{\tau_1,D}(x,t)-K_{\tau,D}(x,t)]\,\rmd t=0, \qquad
                    x\in [0,1],\quad \lambda\in\mathbb C.
\]
Since $\lambda$ and $x$ are arbitrary, $K_{\tau_1,D}=K_{\tau,D}$.

It remains to prove~\emph{(iii)}. If $\tau_1= \tau$, then~$K_{\tau_1,N}=K_{\tau,N}$ by the uniqueness claim of Theorem~\ref{th.21}. Assume that
$K_{\tau_1,N}=K_{\tau,N}$; then, according
to~\eref{eq.21} and~\eref{eq.22},
\[\fl
   \left(\frac{\rmd}{\rmd x}+\tau_1\right)\psi(x,0,\tau_1)
   =\left(\frac{\rmd}{\rmd x}+\tau\right)\psi(x,0,\tau)=0,\qquad
   \psi(\cdot,0,\tau_1)=\psi(\cdot,0,\tau),
\]
and thus $(\tau_1-\tau)\psi(\cdot,0,\tau)=0$. Therefore to prove
part~\emph{(iii)} it is enough to show that the matrix $\psi(x,0,\tau)$
is invertible for every $x\in [0,1]$. Suppose the last statement
is false. Then there exist $x_0\in (0,1]$ and $c\in \mathbb
C^r\setminus\{0\}$ such that $\psi(x_0,0,\tau)c=0$. Therefore the
function $f=\psi(\cdot,0,\tau)c$ is a nonzero solution of the
Cauchy problem $f'+\tau f=0$, $f(x_0)=0$, which is impossible. The proof is complete.
\end{proof}

 For every $\tau\in \mathbf L_2$, $\lambda\in
\mathbb C$ and $x\in [0,1]$ we put
\begin{equation*}\fl
          \mathcal W(x,\lambda,\tau)=\left(\begin{array}{rl}
                                                       \psi(x,\lambda,\tau) & \varphi(x,\lambda,-\tau) \\
                                                       -\varphi(x,\lambda,\tau) & \psi(x,\lambda,-\tau)
                                            \end{array} \right), \quad
           \mathcal Q(x,\lambda,\tau)=\left(\begin{array}{rl}
                                                      -\tau(x) & \lambda I \\
                                                      -\lambda I & \tau(x)
                                            \end{array} \right).
\end{equation*}
It follows from~\eref{eq.21} that $\frac{\rmd}{\rmd x}\mathcal
W(x,\lambda,\tau)=\mathcal Q(x,\lambda,\tau)\mathcal
W(x,\lambda,\tau)$, so that
\[
    \frac{\rmd}{\rmd x} \mathcal W^*(x,\overline\lambda,-\tau^*)
                        \mathcal W(x,\lambda,\tau) =0.
\]
Therefore the matrices $\mathcal W^*(x,\overline\lambda,-\tau^*)$ and $\mathcal W(x,\lambda,\tau)$ are inverse to each other, which yields the relations
\begin{equation}\label{eq.26}
           \eqalign{\varphi(x,\lambda,\tau)\varphi^*(x,\overline\lambda,-\tau^*)
                                 +\psi(x,\lambda,-\tau)\psi^*(x,\overline\lambda,\tau^*)\equiv I, \\
           \eqalign\varphi(x,\lambda,\tau)\psi^*(x,\overline\lambda,-\tau^*)
                                 -\psi(x,\lambda,-\tau)\varphi^*(x,\overline\lambda,\tau^*)\equiv 0.}
\end{equation}

  Let us denote by $\Phi_\tau(\lambda)$ and $\Psi_\tau(\lambda)$
$(\lambda\in\mathbb C)$ the operators acting from $\mathbb C^r$ to
$L_2$ by the formulas
\begin{equation}\label{eq.27}
          [\Phi_\tau(\lambda)c](x):=\sqrt 2\varphi(x,\lambda,\tau)c,\qquad
          [\Psi_\tau(\lambda)c](x):=\sqrt 2\psi(x,\lambda,\tau)c.
\end{equation}
Taking into consideration~\eref{eq.22}, we obtain, for
$\lambda\in\mathbb C$,
\begin{equation}\label{eq.28}
          \Phi_\tau(\lambda)=(\mathcal I+\mathcal K_{\tau,D})\Phi_0(\lambda), \qquad
           \Psi_\tau(\lambda)=(\mathcal I+\mathcal K_{\tau,N})\Psi_0(\lambda),
\end{equation}
where $\mathcal K_{\tau,D}$ and $\mathcal K_{\tau,N}$ are integral
operators with kernels~$K_{\tau,D}$ and~$K_{\tau,D}$ respectively
and~$\mathcal I$ is the identity operator in the algebra $\mathcal
B(L_2)$, which is an algebra of bounded linear operators acting in~$L_2$. Note that since~$K_{\tau,D}$ and $K_{\tau,N}$ belong to~$G_2^+$, the operators~$\mathcal K_{\tau,D}$ and~$\mathcal K_{\tau,N}$ belong to the algebra~$\mathfrak G_2^+$
(see~\ref{add.1}), and hence they are Volterra operators~\cite[Ch.~IV]{GGK1}.

\begin{lemma}\label{le.23}
Let $\tau \in \mathbf L _2$. Then the following statements hold:
\begin{itemize}
\item[(i)] the operator functions $\lambda\mapsto
\Phi_\tau(\lambda)/\lambda$ and $\lambda\mapsto
\Psi_\tau(\lambda)$ are analytic in $\mathbb C$; moreover,
\begin{equation}\label{eq.29}
         \eqalign{\ker\Phi_\tau(\lambda)=\{0\}, \qquad
                       \Ran\Phi^*_\tau(\lambda)=\mathbb C^r,  \qquad
                       \lambda\in\mathbb C\setminus\{0\},      \\
         \eqalign \ker\Psi_\tau(\lambda)=\{0\},  \qquad
                       \Ran\Psi^*_\tau(\lambda)=\mathbb C^r,   \qquad
                                                \lambda\in\mathbb C,}
\end{equation}
\begin{equation}\fl \label{eq.210}
       \ker(T_q -\lambda^2\mathcal I)=\Phi_\tau(\lambda)\mathcal E_\lambda, \quad
       \ker(S_\tau -\lambda^2\mathcal I)=\Psi_\tau(\lambda)\mathcal E_\lambda,\qquad \lambda\in\mathbb C,
\end{equation}
where $\mathcal E_\lambda:=\ker\varphi(1,\lambda,\tau)$;

\item[(ii)] the operator functions $\lambda\mapsto
\varphi(1,\lambda,\tau)^{-1}$ and
\begin{equation}\label{eq.211}
         \lambda\mapsto m_\tau(\lambda)=-\varphi(1,\lambda,\tau)^{-1}\psi(1,\lambda,-\tau)
\end{equation}
are meromorphic in $\mathbb C$. Moreover,
$m_0(\lambda)=-\cot\lambda I$ and
\begin{equation}\label{eq.212}
          \|m_\tau(\lambda) + \cot\lambda I\|=\mathrm{o}(1)
\end{equation}
as $\lambda\to\infty$ within the domain $\mathcal O =\{z\in\mathbb
C\mid\forall n\in\mathbb Z \quad |z-\pi n|>1 \}$.
\end{itemize}
\end{lemma}

\begin{proof} Part \emph{(i)} obviously follows from~\eref{eq.27} and~\eref{eq.28}.
From~\eref{eq.22} we obtain that
\begin{equation}\label{eq.213}
          \eqalign{ \varphi(1,\lambda,\tau)=\sin\lambda I+\int_{0}^{1}(\sin\lambda t) K_{\tau,{D}}(1,t)\,\rmd t, \\
           \eqalign \psi(1,\lambda,-\tau)=\cos \lambda I+\int_{0}^{1}(\cos\lambda t)\,K_{-\tau,N}(1,t)\,\rmd t.}
\end{equation}
Since  $K_{\tau,D},K_{-\tau,N}\in G_2^+$, the functions
$K_{\tau,D}(1,\cdot)$ and $K_{-\tau,N}(1,\cdot)$ belong to
$\mathbf L_2$. The relations
\begin{equation*}
         \lim\limits_{|\lambda|\to\infty}\rme^{-|\Im\lambda|}\int_0^1 f(x)\cos\lambda x \,\rmd x
        =\lim\limits_{|\lambda|\to\infty}\rme^{-|\Im\lambda|}\int_0^1 f(x)\sin\lambda x \,\rmd x
        =0
\end{equation*}
for $f\in L_1(0,1)$ that refine the classical Riemann--Lebesgue theorem~\cite[Ch.1]{Titch} are proved in~\cite[Ch.1]{Mar}; thus
\[ \fl
        \lim\limits_{|\lambda|\to\infty}\rme^{-|\Im\lambda|}\|\varphi(1,\lambda,\tau)-\sin\lambda I\|
    =\lim\limits_{|\lambda|\to\infty}\rme^{-|\Im\lambda|}\|\psi(1,\lambda,-\tau)-\cos\lambda I\|
    =0.
\]
In particular, $\varphi(1,\lambda,\tau)$ is invertible for all $\lambda\in\mathcal O$ large enough, so that $m_\tau$ is meromorphic and the relation~\eref{eq.212} holds.
\end{proof}

\begin{theorem}\label{th.24}
Let $\tau\in \Re\mathbf L _2$ and $q=b(\tau)$. Then the following
statements hold:
  \begin{enumerate}
     \item The operators $S_\tau$ and $T_q$ are self--adjoint; moreover, $S_\tau\ge 0$ and $T_q>0$.
     \item The spectra  $\sigma(S_\tau)$ and $\sigma(T_q)$ consist of isolated eigenvalues and
                  \begin{equation}\label{eq.214}
                             \sigma(S_{\tau})=\{\lambda^2 \mid\ker\varphi(1,\lambda,\tau)\ne\{0\}\}=\sigma(T_q)\cup\{0\}.
                  \end{equation}
     \item  Let $\lambda_{j}=\lambda_j(\tau)$ and $P_{j,\tau}$ $($resp. $Q_{j,q})$ be the orthogonal projector on the eigensubspace $\ker(S_\tau-\lambda^2_j\mathcal I)$ $($resp. $\ker(T_q-\lambda^2_j\mathcal I))$; then
                   \begin{equation}\label{eq.215}
                             \sum\limits_{j=0}^\infty P_{j,\tau} =\mathcal I, \qquad
                             \sum\limits_{k=1}^\infty Q_{k,q} =\mathcal I.
                   \end{equation}
     \item The norming constants $\alpha_j=\alpha_j(\tau)$ (see Introduction) satisfy the relations $\alpha_0>0$ and $\alpha_j\ge 0$, $j\in\mathbb N$. Moreover, for $j\in\mathbb Z_+$ and $k\in\mathbb N$ we have
                  \begin{equation}\label{eq.216}
                             P_{j,\tau}=\Psi_\tau(\lambda_j)\alpha_j\Psi^*_\tau(\lambda_j), \qquad
                             Q_{k,q}=\Phi_\tau(\lambda_k)\alpha_k\Phi^*_\tau(\lambda_k),
                  \end{equation}
where $\Phi^*_\tau(\lambda)=[\Phi_\tau(\lambda)]^*$,
$\Psi^*_\tau(\lambda)=[\Psi_\tau(\lambda)]^*$ (see~\eref{eq.27}).
     \item  If $\tau_1\in \Re\mathbf L _2$ and $b(\tau_1)=b(\tau)$, then $\lambda_j(\tau_1)=\lambda_j(\tau)$ and $\alpha_j(\tau_1)=\alpha_j(\tau)$ for every $j\in\mathbb N$.
  \end{enumerate}
\end{theorem}

\begin{proof}
Writing $S=S_\tau$ and $T=T_q$ for short and integrating by parts, we get
\[
  (Sf|g)=\left(f'-\tau f|g'-\tau g\right)=(f|Sg),\qquad f,g\in \mathit{D}(S),
\]
\[
  (Tf|g)=\left(f'-\tau f|g'-\tau g\right)=(f|Tg),\qquad f,g\in \mathit{D}(T),
\]
where $(\,\cdot\,|\,\cdot\,)$ is the scalar product in the space
$L_2$. Therefore the operators $S$ and $T$ are symmetric and
nonnegative. Suppose that
 $(Tf|f)= 0$ for some $f\in \mathit{D}(T)$. Then~$f$ is a solution of the Cauchy problem
$f'-\tau f=0$, $f(0)=0$. The uniqueness theorem then gives $f=0$,
and, therefore $T>0$.

Take now an arbitrary~$f \in L_2$ and a point~$\lambda\in\mathbb C\setminus\{0\}$ for which the matrix $\varphi(1,\lambda,\tau)$ is nonsingular. The function
\begin{equation*}
    \eqalign{g(x) = [X(\lambda)f](x) &:=
        \frac{\varphi(x,\lambda,\tau)}\lambda
            \int_x^1 \psi^*(t,\overline\lambda,-\tau)\, f(t)\,\rmd t \\
      &\quad + \frac{\psi(x,\lambda,-\tau)}\lambda
            \int_0^x \varphi^*(t,\overline\lambda,\tau)\, f(t)\,\rmd t}
\end{equation*}
vanishes at $x=0$, belongs to the domain of the differential expression~$\mathfrak t_\tau$, and solves the equation $\mathfrak t_\tau g = \lambda^2 g + f$, as can be verified directly by using the relations~\eref{eq.21} and~\eref{eq.26}. A generic solution of the above differential equation that vanishes at $x=0$ takes therefore the form
$\varphi(x,\lambda,\tau) c + g$ for some vector $c\in\mathbb C^r$; the choice
\[
    c:= \frac{m_\tau(\lambda)}{\lambda}\int_0^1 \varphi^*(t,\overline\lambda,\tau)\, f(t)\,\rmd t
\]
makes this solution to vanish at the point $x=1$ as well. This implies that the point $\lambda^2$ is a resolvent point of the operator $T$ and that the resolvent of $T$ is given by
\[
    (T-\lambda^2\mathcal I)^{-1} = \frac1{2\lambda}\Phi_\tau(\lambda) m_\tau(\lambda)\Phi_\tau^*(\overline\lambda)+ X(\lambda).
\]
Similar arguments show that the resolvent of the operator $S$ at the point~$\lambda^2$ equals
\[
    (S-\lambda^2\mathcal I)^{-1} = \frac1{2\lambda}\Psi_\tau(\lambda)
        m_\tau(\lambda)\Psi_\tau^*(\overline\lambda)
        - Y(\lambda),
\]
where the operator $Y$ is given via
\[\fl
    [Y(\lambda)f](x)= \frac{\varphi(x,\lambda,-\tau)}{\lambda}
        \int_0^x \psi^*(t,\overline\lambda,\tau)\, f(t)\,\rmd t +
        \frac{\psi(x,\lambda,\tau)}{\lambda}
        \int_x^1\varphi^*(t,\overline\lambda,-\tau)\, f(t)\,\rmd t.
\]
The above formulas show that $S$ and $T$ have compact resolvents
and thus $(i)$--$(iii)$ follow.

Recall that $-\alpha_j(\tau)$ is the residue of the
Weyl--Titschmarsh function $m_\tau$ at the
point~$\lambda_j(\tau)$, $j\in \mathbb{N}$. Taking $\varepsilon>0$
small enough, we get
\[ \eqalign{
    P_{j,\tau}
        &=\frac{-1}{2\pi \rmi} \oint_{|\zeta-\lambda_j^2|=\varepsilon}
            (S-\zeta\mathcal I)^{-1}\,\rmd \zeta
         =\frac{-1}{2\pi \rmi}\oint_{|\lambda-\lambda_j|=\varepsilon}
            2\lambda(S-\lambda^2\mathcal I)^{-1}\,\rmd \lambda\\
        &=\frac{-1}{2\pi \rmi}\oint_{|\lambda-\lambda_j|=\varepsilon}
            \Psi_\tau(\lambda)m_\tau(\lambda) \Psi^*_\tau(\lambda)\,\rmd \lambda
        =\Psi_\tau(\lambda_j)\alpha_j\Psi^*_\tau(\lambda_j)}
\]
for every $j\in\mathbb N$. Similarly, we obtain that
\begin{equation*}
   P_{0,\tau}=\Psi_\tau(0)\alpha_0\Psi^*_\tau(0), \qquad
   Q_{j,q}=\Phi_\tau(\lambda_j)\alpha_j\Phi^*_\tau(\lambda_j), \qquad
   j\in\mathbb N.
\end{equation*}
Recalling~\eref{eq.29}, we see that $\alpha_j\ge 0$ for all
$j\in \mathbb Z_+$.  Let us prove that $\alpha_0>0$.
Assume the contrary; then
\[
    \ker S=\Ran P_{0,\tau}=\Ran[\Psi_\tau(0)\alpha_0\Psi^*_\tau(0)]\ne \Ran\Psi_\tau(0).
\]
On the other hand, \eref{eq.210} on account of the equality
$\ker\varphi(1,0,\tau)=\mathbb C^r$ yields $\ker
S=\Ran\Psi_\tau(0)$. The contradiction derived shows that
$\alpha_0>0$.

It remains to prove~$(v)$. Let $\tau_1\in\Re\mathbf L _2$ be
such that $b(\tau_1)=b(\tau)$. It follows from Lemma~\ref{le.22}
that $K_{\tau_1,D}=K_{\tau,D}$, and thus
$\varphi(\cdot,\cdot,\tau_1)=\varphi(\cdot,\cdot,\tau)$.
Therefore, in view of~\eref{eq.214} we get
$\lambda_j(\tau_1)=\lambda_j(\tau)$ for all $j\in\mathbb N$ and,
moreover, $\Phi_{\tau_1}(\cdot)=\Phi_{\tau}(\cdot)$. It follows
now from~\eref{eq.216} that
\[
   \Phi_\tau(\lambda_j)\alpha_j(\tau_1)\Phi^*_\tau(\lambda_j)=Q_{j,q}
  =\Phi_\tau(\lambda_j)\alpha_j(\tau)\Phi^*_\tau(\lambda_j), \qquad
   j\in\mathbb N.
\]
Hence, using~\eref{eq.29}, we obtain that
$\alpha_j(\tau_1)=\alpha_j(\tau)$ for all $j\in\mathbb N$. The
proof is complete.
\end{proof}

\subsection{{The asymptotics of eigenvalues and norming constants}}
\label{subsec.22} The main result of this section is following
theorem.
\begin{theorem}\label{th.25}
Let $\tau\in \Re\mathbf L _2$. Then for the sequence $\mathfrak
a=\mathfrak a_\tau$ the condition $(A_1)$ holds.
\end{theorem}

First we prove two lemmas. In the sequel, we shall use the
following notation. If $(\lambda_{j})_{j\in\mathbb Z_+}$ is a
strictly increasing sequence of nonnegative numbers and
$(\alpha_j)_{j\in\mathbb Z_+}$ is a sequence in $M^+_r$, then
\begin{equation}\label{eq.218}
          \beta_n:= I-\sum\limits_{\lambda_k\in\Delta_n}\alpha_k, \qquad
          \widetilde\lambda_j:=\lambda_j-\pi n, \qquad \lambda_j\in\Delta_n,\quad
          n\in\mathbb N,
\end{equation}
with $\Delta_n$ defined in Subsection~\ref{subsec.13}.

\begin{lemma}\label{le.26}
Let $\tau\in \Re\mathbf L _2$ and $\lambda_j=\lambda_j(\tau)$,
$j\in\mathbb Z_+$. Then
\begin{equation}\label{eq.219}
          \sum_{n=1}^{\infty}\sum\limits_{\lambda_j\in\Delta_n}|\widetilde\lambda_j|^2<\infty, \qquad
          \sup\limits_{n\in\mathbb N}\sum\limits_{\lambda_j\in\Delta_n} 1<\infty.
\end{equation}
\end{lemma}

\begin{proof}
Let us fix $\tau\in \Re\mathbf L _2$ and note that the
numbers~$\lambda_j(\tau)$, ${j\in\mathbb Z_+}$, are nonnegative
zeros of an entire
function~$g(\lambda):=\det\varphi(1,\lambda,\tau)$,
$\lambda\in\mathbb C$. In view of~\eref{eq.213} the function $g$
is odd. Taking into account~\eref{eq.214}, we conclude that zeros
of the function $g$ are all real. The function  $\lambda\mapsto
 \varphi(1,\lambda,\tau)$ belongs to the following class of function $ \mathbb C\to
 M_r$:
\begin{equation*}
        \mathcal{F}_f(\lambda):=\sin{\lambda}I +\int_{-1}^1 f(t)\rme^{\rmi\lambda t}\,\rmd t, \qquad
        \lambda\in\mathbb C,
\end{equation*}
where $f\in L_2((-1,1),M_r)$. Indeed, it follows
from~\eref{eq.213} that $\varphi(1,\cdot,\tau)=\mathcal{F}_{f}$
for $f(t)=(\sign t) K_{\tau,D}(1,|t|)/(2\rmi)$. It is shown in the
paper~\cite{Tr} that the set of zeros of a function $\det
\mathcal{F}_f$ with ${\mathcal F}_f$ as above can be indexed
(counting multiplicities) by the set $\mathbb Z$ so that the
corresponding sequence $(\omega_n)_{n\in\mathbb Z}$ has the
asymptotics
\begin{equation*}
         \omega_{kr+j}=\pi k +\widehat\omega_{j,k},\qquad
         k\in\mathbb{Z},\quad
         j=0,\dots,r-1,
\end{equation*}
where the sequences $(\widehat\omega_{j,k})_{k\in\mathbb{Z}}$
belong to $\ell_2(\mathbb Z)$. Therefore,~\eref{eq.219} follows,
and the proof is complete.
\end{proof}

\begin{lemma}\label{le.27}
Let the operators $A(z)$ and $B(z)$ $(z\in\mathbb C)$ act from
$\mathbf L_2$ to $M_r$ by the formulas
\begin{equation}\label{eq.221AB}
              A(z)f=\sqrt2\int_{0}^{1} (\sin zt) f(t)\,\rmd t, \qquad
              B(z)g=\sqrt2\int_{0}^{1} (\cos zt) g(t)\,\rmd t.
\end{equation}
Then for every  $f_1,f_2\in\mathbf L_2$ and $\lambda\in \mathbb
D:=\{\lambda\in\mathbb C \mid |\lambda|\le 1\}$ we have
\begin{equation}\label{eq.220}
          \sum\limits_{n=1}^\infty (\|A(\pi n +\lambda)f_1\|^2
          +\|B(\pi n +\lambda)f_2\|^2)
           \le 16r^2(\|f_1\|^2_{\mathbf L_2}+\|f_2\|^2_{\mathbf L_2}).
\end{equation}
\end{lemma}

\begin{proof}
Since  $\{\sqrt 2\sin\pi nt\}_{n\in\mathbb N}$ and $\{\sqrt
2\cos\pi nt\}_{n\in\mathbb N}$ form orthonormal systems of
$L_2(0,1)$, it follows that
\begin{equation}\label{eq.222d}
            \sum\limits_{n=1}^\infty (\|A(\pi n)f_1\|^2+\|B(\pi n)f_2\|^2)
            \le 2r^2(\|f_1\|^2_{\mathbf L_2}+\|f_2\|^2_{\mathbf L_2}).
\end{equation}
Relations~\eref{eq.221AB} yield
\begin{equation}\label{eq.223}\fl
           A(\pi n+\lambda)f_1=B(\pi n)f_{11}+A(\pi n)f_{12},  \qquad
           B(\pi n +\lambda)f_2= B(\pi n)f_{22} -A(\pi n)f_{21},
\end{equation}
where
\[
        f_{j1}(t)=(\sin\lambda t) f_j(t), \qquad
        f_{j2}(t)=(\cos\lambda t)f_j(t),  \qquad t\in (0,1), \quad j=1,2.
\]
Since $|\sin\lambda|^2+|\cos\lambda|^2\le (\sinh1 +\cosh1)^2\le 8$
for $\lambda\in\mathbb D$, we have
\[
     \|f_{j1}\|^2_{\mathbf L_2} +\|f_{j2}\|^2_{\mathbf L_2} \le 8\|f_{j}\|^2_{\mathbf L_2},\qquad j=1,2.
\]
Hence, using~\eref{eq.223} and~\eref{eq.222d}, we arrive
at~\eref{eq.220}.
\end{proof}

\begin{proofof}{Proof of Theorem~\ref{th.25}}
Let $\tau\in \Re\mathbf L _2$ and $\lambda_j=\lambda_j(\tau),
\alpha_j=\alpha_j(\tau)$. According to Lemma~\ref{le.26}, it
remains to show that
$\sum\limits_{n=1}^{\infty}\|\beta_n\|^2<\infty$. Set
\[
   \mathbb T_n:=\{z\in\mathbb C  \mid  |\lambda -\pi n|=1 \}, \qquad
   \mathbb D_n:=\{z\in\mathbb C  \mid  |\lambda -\pi n|\le 1 \}, \quad
    n\in\mathbb Z_+,
\]
\[
    f_1:=2^{-1/2}K_{\tau,D}(1,\cdot), \qquad
    f_2:=2^{-1/2}K_{-\tau,N}(1,\cdot).
\]
 Since $K_{\tau,D},K_{-\tau,N}\in G_2^+$, we have $f_1,f_2\in \mathbf L_2$.
It follows from~\eref{eq.213} and~\eref{eq.221AB}
 that
\begin{equation*}
           \varphi(1,\lambda,\tau)=\sin\lambda I +  A(\lambda)f_1, \qquad
           \psi(1,\lambda,-\tau)=\cos \lambda I + B(\lambda)f_2.
\end{equation*}
Hence, using~\eref{eq.211}, we obtain the equality
\begin{equation}\label{eq.224}
          m_\tau(\lambda)+\cot(\lambda)I=\varphi(1,\lambda, \tau)^{-1}[(\cot\lambda) A(\lambda)f_1 -B(\lambda)f_2].
\end{equation}

In view of~\eref{eq.219} and~\eref{eq.221AB}, we can choose
$n_0\in\mathbb N$, such that
\begin{equation}\label{eq.225}
          \sum_{n=n_0}^{\infty}\sum\limits_{\lambda_j\in\Delta_n}|\widetilde\lambda_j|^2<1, \qquad
          \sup\limits_{n\ge n_0}\sup\limits_{\lambda\in\mathbb D_n}\|A(\lambda)f_1\| \le\frac14.
\end{equation}
Note that for every $\lambda\in \mathbb T_0$
\[
  |\sin\lambda|\ge 1-\sum\limits_{n=1}^\infty \frac1{(2n+1)!!}\ge 1 -\sum\limits_{n=1}^\infty \frac1{3^n}=\frac12,
\]
whence $|\sin\lambda|\ge \frac12$ and $|\cot\lambda|\le\sqrt 3$ for
$\lambda\in \mathbb T_n$, $n\in\mathbb Z_+$. Therefore estimates~\eref{eq.225} yield the following inequalities for $\lambda\in\mathbb T_n$ with~$n\ge n_0$:
\begin{equation}\label{eq.226}
           \|\varphi(1,\lambda, \tau)^{-1}\|\le|\sin\lambda|^{-1}(1-|\sin\lambda|^{-1}\|A(\lambda)f_1\|)^{-1}\le 4.
\end{equation}
For $n\ge n_0$ the function $m_\tau$ has no poles on the
circle~$\mathbb T_n$ and $\{\lambda_j \mid \lambda_j\in\Delta_n\}
\subset \mathbb D_n$. Hence~\eref{eq.15} implies that
\begin{equation}\label{eq.227}
         \beta_n=\frac{1}{2\pi \rmi}\oint_{\mathbb T_n}
         (m_\tau(\lambda)+\cot \lambda I)\,\rmd\lambda,
         \qquad n\ge n_0.
\end{equation}
Using~\eref{eq.224} and~\eref{eq.226}, we get
\[
    \|m_\tau(\lambda)+\cot\lambda I\|^2
    \le 96(\|A(\lambda)f_1\|^2+\|B(\lambda)f_2\|^2), \quad
    \lambda\in\mathbb T_n,\quad n\ge n_0.
\]
It follows now from~\eref{eq.227} that
\[
    \|\beta_n\|^2\le\frac{96}{2\pi}\int_{\mathbb T_0}
    \{\|A(\pi n +\rme ^{\rmi t})f_1\|^2+
\|B(\pi n +\rme^{\rmi t})f_2\|^2\}\,\rmd t, \qquad n\ge n_0,
\]
and in view of~\eref{eq.220} we get  $\sum_{n=1}^{\infty}\|\beta_n\|^2 <\infty$.
\end{proofof}

\subsection{The Weyl--Titchmarsh function}\label{subsec.23}

\begin{proposition}\label{pr.210}
Let $\tau\in \Re\mathbf L _2$. Then $m_\tau$ is a Herglotz
function and
\begin{equation}\label{eq.228}
            m_\tau(\lambda)=2\lambda\int_0^{\infty}\frac{\rmd \nu_\tau(\xi)}{\xi^2-\lambda^2}, \qquad \lambda\in\mathbb C.
\end{equation}
\end{proposition}

\begin{proof} That $m_\tau$ is Herglotz is clear, so we only need to prove~\eref{eq.228}. Set
\[
    h(\lambda):=m_{\tau}(\lambda)-2\lambda\int_0^{\infty}\frac{\rmd \nu_\tau(\xi)}{\xi^2-\lambda^2},  \qquad
    \lambda\in\mathbb C.
\]
It follows from~\eref{eq.15} and~\eref{eq.16} that  $h$ is
entire. To show that $h=0$ it suffices to justify the estimate $\|h(\lambda)\|=o(1)$ as $\lambda$ goes to infinity within the domain
\[
    \mathcal O =\{z\in\mathbb C \mid \forall n\in\mathbb Z\quad |z-\pi n|>1 \}.
\]
We set
\[
     g(\lambda):=\cot\lambda I+2\lambda\int_0^{\infty}\frac{\rmd
     \nu_\tau(\xi)}{\xi^2-\lambda^2}, \qquad \lambda\in\mathbb C.
\]
In virtue of~\eref{eq.212}, it suffices to prove the estimate
\begin{equation}\label{eq.229}
          \|g(\lambda)\|=o(1), \qquad \mathcal O\ni\lambda\to\infty.
\end{equation}
Since
\[
    \cot\lambda=\frac1{\lambda}+\sum\limits_{n=1}^{\infty}\frac{2\lambda}{\lambda^2-(\pi n)^2},  \qquad
    \lambda\in\mathbb C,
\]
we find that
\begin{equation*}
           g(\lambda)=\frac 1{\lambda}(I-2\alpha_0)
           +\sum_{n=1}^{\infty}\frac{2\lambda I}{\lambda^2-(\pi n)^2}
           -\sum_{n=1}^{\infty}\sum\limits_{\lambda_j\in\Delta_n}\frac{2\lambda\alpha_j}{\lambda^2-\lambda_j^2}.
\end{equation*}
Let
\begin{equation*}
           g_n(\lambda)=\frac{2\lambda I }{\lambda^2-(\pi n)^2}
           -\sum\limits_{\lambda_j\in\Delta_n}\frac{2\lambda\alpha_j}{\lambda^2-\lambda_j^2}, \qquad
           n\in\mathbb N.
\end{equation*}
It is easy to verify that for arbitrary  $\lambda_j\in\Delta_n$
such that $|\widetilde\lambda_j|\le 1/2,$
\begin{equation*}
           \left|\frac{2\lambda}{\lambda^2-(\pi n)^2}-\frac{2\lambda}{\lambda^2-\lambda_j^2}\right|
           \le\frac{2|\widetilde\lambda_j|}{|\lambda-\pi n|}+\frac{2|\widetilde\lambda_j|}{|\lambda+\pi n|},\qquad
            \lambda\in \mathcal O.
\end{equation*}
Therefore, for big enough $n$ and $\lambda\in \mathcal O$,
\begin{equation*}
    \|g_n(\lambda)\|
           \le\Biggl(\|\beta_n\|
                +\sum\limits_{\lambda_j\in\Delta_n}2
                    |\widetilde\lambda_j|\|\alpha_j\|\Biggr)
           \left(\frac1{|\lambda-\pi n|}+\frac1{|\lambda+\pi n|}\right).
\end{equation*}
Taking into account Theorem~\ref{th.25} and using the
Cauchy--Bunyakowski inequality, we get that there exists a
constant $C$ such that, for all big enough $k$ and
$\lambda\in\mathcal O$,
\begin{equation}\label{eq.230}\fl
           \sum\limits_{n=k}^\infty \|g_n(\lambda)\|\le C\Biggl(\sum\limits_{n=k}^\infty\|\beta_n\|^2
           +\sum\limits_{n=k}^\infty\sum\limits_{\lambda_j\in\Delta_n}
            |\widetilde\lambda_j|^2\Biggr)^{1/2}
           \Biggl(\sum\limits_{n\in\mathbb Z} \frac1{|\lambda-\pi n|^2}\Biggr)^{1/2}.
\end{equation}
Since
\begin{equation*}
          \sum\limits_{n\in\mathbb Z}\frac1{|\lambda-\pi n|^2}\le2+\frac2{\pi^2}\sum\limits_{n=1}^\infty \frac1{n^2}\le 3,\qquad
           \lambda\in \mathcal O,
\end{equation*}
and $\|g_n(\lambda)\|=O(1/\lambda)$ as $\lambda\to \infty$,
\eref{eq.229} follows from~\eref{eq.230}. The proof is complete.
\end{proof}

\subsection{{Necessity parts of Theorems~\ref{th.14}  and~\ref{th.15}}}
\label{subsec.24}  The theorem stated below implies necessity in
Theorems~\ref{th.14} and~\ref{th.15}.
\begin{theorem}\label{th.29}
Let $\tau\in \Re\mathbf L _2$. Then for a sequence $\mathfrak
a_\tau$ the conditions $(A_1)$--$(A_4)$ hold.
\end{theorem}

First we prove four auxiliary but nonetheless important lemmas.

\begin{lemma}\label{le.210}
Assume that $\tau\in \mathbf L _2$ and let $\mathfrak
a=((\lambda_j,\alpha_j))_{j\in\mathbb Z_+}$ satisfy
condition~$(A_1)$. Then
\begin{equation}\label{eq.231}\fl
          \sum_{n=1}^\infty\sum_{\lambda_j\in\Delta_n}\|\Psi_\tau(\lambda_j) -\Psi_0(\pi n)\|^2<\infty, \qquad
          \sum_{n=1}^\infty\sum_{\lambda_j\in\Delta_n}\|\Phi_\tau(\lambda_j)-\Phi_0(\pi n)\|^2<\infty.
\end{equation}
\end{lemma}

\begin{proof} We prove the first inequality; the second one can be proved similarly.
Recalling the definitions in~\eref{eq.27}--\eref{eq.28}, we conclude that
for every $\lambda$ and $\xi$ in~$\mathbb R$ one has
\begin{equation*}
          \|\Psi_\tau(\lambda)\|\le C, \qquad
          \|\Psi_\tau(\lambda)-\Psi_\tau(\xi)\|\le C|\lambda-\xi|
\end{equation*}
with $C=\sqrt2(1+\|\mathscr K_{\tau,N}\|)$. By virtue of~\eref{eq.218} we derive the estimate
\begin{equation}\label{eq.233}
          \sum_{n=1}^\infty\sum_{\lambda_j\in\Delta_n}\|\Psi_\tau(\lambda_j)-\Psi_\tau(\pi n)\|^2
          \le C^2\sum_{n=1}^\infty|\widetilde\lambda_j|^2 <\infty.
\end{equation}
We note that the operator $\mathscr K_{\tau,N}$ belongs to the
Hilbert--Schmidt class $\mathcal B_2$ and that the sequence
$(P_{n,0})_{n\in\mathbb Z_+}$ consists of pairwise orthogonal
projectors. Therefore,
\begin{equation*}
          \sum_{n=1}^\infty\|\mathscr K_{\tau,N} P_{n,0}\|^2
          \le\sum_{n=1}^\infty \|\mathscr K_{\tau,N} P_{n,0}\|_{\mathcal B_2}^2 \le \|\mathscr K_{\tau,N}\|_{\mathcal B_2}^2.
\end{equation*}
It follows from~\eref{eq.28} that $\Psi_\tau(\pi n)-\Psi_0(\pi
n)=\mathscr K_{\tau,N}\Psi_0(\pi n)$. Since $\|\Psi_0(\pi n)\|=1$
and $P_{n,0}\Psi_0(\pi n)=\Psi_0(\pi n)$,
\begin{equation}\label{eq.234}
           \sum_{n=1}^\infty \|\Psi_\tau(\pi n)-\Psi_0(\pi n)\|^2
           \le\sum_{n=1}^\infty \|\mathscr K_{\tau,N}P_{n,0}\Psi_0(\pi n)\|^2
           \le\|\mathscr K_{\tau,N} \|_{\mathcal B_2}^2.
\end{equation}
Combining~\eref{eq.233} and~\eref{eq.234} with the uniform bound~$\sup\limits_{n\in\mathbb N}\sum\limits_{\lambda_j\in\Delta_n} 1 <\infty$ satisfied due to~$(A_1)$, we obtain the inequality~\eref{eq.231} for the operators~$\Psi_\tau$.
\end{proof}

\begin{lemma}\label{le.211}
Assume that $\tau\in \mathbf L _2$ and let $\mathfrak
a=((\lambda_j,\alpha_j))_{j\in\mathbb Z_+}$ satisfy
condition~$(A_1)$. Then for every  $f\in L_2$ we get
 \begin{equation}\label{eq.235}
          \sum_{j=0}^\infty\|\alpha_j\Psi^*_\tau(\lambda_j)f\|_{\mathbb C^r}^2<\infty, \qquad
          \sum_{j=1}^\infty\|\alpha_j\Phi^*_\tau(\lambda_j)f\|_{\mathbb C^r}^2<\infty.
 \end{equation}
Furthermore, for every sequence $(c_j)\in \ell_2(\mathbb Z_+,\mathbb C^r)$ the series $\sum_{j=0}^\infty \Psi_\tau(\lambda_j)c_j$ and $\sum_{j=1}^\infty
\Phi_\tau(\lambda_j)c_j$ converge in the space $L_2$.
\end{lemma}

\begin{proof}
We shall only prove the first inequality in~\eref{eq.235} and convergence of the series~$\sum_{j=0}^\infty \Psi_\tau(\lambda_j)c_j$; the other statements are justified analogously. Since the norms of the matrices~$\alpha_j$ are uniformly bounded, it suffices to prove the inequality
\(
\sum_{j=0}^\infty\|\Psi^*_\tau(\lambda_j)f\|_{\mathbb C^r}^2 < \infty.
\)
Observe that
\[
   \eqalign{ \sum_{j=0}^\infty\|\Psi^*_\tau(\lambda_j)f\|_{\mathbb C^r}^2
        &= \sum_{n=0}^\infty \sum_{\lambda_j\in \Delta_n}
            \|\Psi^*_\tau(\lambda_j)f\|_{\mathbb C^r}^2\\
        &\le  2 \sum_{n=0}^\infty \sum_{\lambda_j\in \Delta_n}
            \Bigl\{
             \|[\Psi^*_\tau(\lambda_j)-\Psi^*_0(\pi n)]f\|_{\mathbb C^r}^2
              + \|\Psi^*_0(\pi n)f\|_{\mathbb C^r}^2\Bigr\}}
\]
and that
\[
    \sum_{n=1}^\infty \|\Psi^*_0(\pi n)f\|_{\mathbb C^r}^2
        \le \|f\|_{\mathbb C^r}^2
\]
by the Bessel inequality for the orthonormal system~$\{\sqrt2 \cos\pi nx\}_{n\in\bN}$. The desired inequality follows now from Lemma~\ref{le.210} and the fact that, by virtue of~$(A_1)$, $\sup\limits_{n\in\mathbb N}\sum\limits_{\lambda_j\in\Delta_n}
1 <\infty$.

To prove convergence of the series $\sum_{j=0}^\infty \Psi_\tau(\lambda_j)c_j$, we similarly write
\[
    \sum_{j=0}^\infty \Psi_\tau(\lambda_j)c_j
        = \sum_{n=0}^\infty \sum_{\lambda_j\in \Delta_n}
              \bigl[\Psi_\tau(\lambda_j) - \Psi_0(\pi n)\bigr]c_j
        + \sum_{n=0}^\infty \sum_{\lambda_j\in \Delta_n}
                \Psi_0(\pi n)c_j.
\]
Since~$(c_j)\in \ell_2(\mathbb Z_+,\mathbb C^r)$, the first series converges in~$L_2$ due to~\eref{eq.231} and the Cauchy--Bunyakowski inequality, while the second one due to the fact that the system~$\{\sqrt2 \cos\pi nx\}_{n\in\bN}$ is orthonormal in $L_2(0,1)$.
\end{proof}

\begin{lemma}\label{le.212}
Assume that $\tau\in \mathbf L _2$, let $\mathfrak
a=((\lambda_j,\alpha_j))_{j\in\mathbb Z_+}$ satisfy
condition~$(A_1)$, and set
\begin{equation*}
           \widehat P_j=\Psi_\tau(\lambda_j)\alpha_j\Psi^*_\tau(\lambda_j), \qquad
           \widehat Q_j=\Phi_\tau(\lambda_j)\alpha_j\Phi^*_\tau(\lambda_j), \qquad
           j\in\mathbb N.
\end{equation*}
Then the series  $\sum_{j=1}^\infty \widehat P_j$ and
$\sum_{j=1}^\infty \widehat Q_j$ converge in the strong operator
topology and, moreover,
\begin{equation}\label{eq.237}
           \sum_{n=1}^\infty\|P_{n,0}-\sum_{\lambda_j\in\Delta_n}\widehat P_j \|^2 <\infty, \qquad
           \sum_{n=1}^\infty\|Q_{n,0}-\sum_{\lambda_j\in\Delta_n}\widehat Q_j \|^2 <\infty.
\end{equation}
\end{lemma}

\begin{proof}
Convergence of the series $\sum_{j=1}^\infty \widehat P_j$
and $\sum_{j=1}^\infty \widehat Q_j$ follows
from Lemma~\ref{le.211}. We prove the first inequality in~\eref{eq.237}; the proof of the second one is similar.
Recall that $P_{n,0} = \Psi_0(\pi n)\Psi^*_0(\pi n)$ and that (see~\eref{eq.218})
\[
    \sum_{\lambda_j\in\Delta_n} \Psi_0(\pi n)\alpha_j\Psi^*_0(\pi n)
        = P_{n,0} - \Psi_0(\pi n)\beta_n\Psi^*_0(\pi n).
\]
Using this and the relation
 \begin{equation*}\fl
         \widehat P_j-\Psi_0(\pi n)\alpha_j\Psi^*_0(\pi n)
         = \Psi_\tau(\lambda_j)\alpha_j[\Psi^*_\tau(\lambda_j)-\Psi^*_0(\pi n)] + [\Psi_\tau(\lambda_j)-\Psi_0(\pi n)]\alpha_j\Psi^*_0(\pi n),
\end{equation*}
one concludes that
\[
  \|P_{n,0}-\sum_{\lambda_j\in\Delta_n}\widehat P_j \|^2
  \le C_1\|\beta_n\|^2 +C_2\sum_{\lambda_j\in\Delta_n}\|\Psi_\tau(\lambda_j)- \Psi_0(\pi n)\|^2 ,
\]
where $C_1$ and $C_2$ are positive constants independent
of~$n\in\mathbb N$. The result now follows from~\eref{eq.231} and
assumption~$(A_1)$.
\end{proof}

Let $\tau\in\mathbf L _2$ and let $\frak
a=((\lambda_j,\alpha_j))_{j\in\mathbb Z_+}$ satisfy
condition~$(A_1)$. We denote by $\mathscr U^e_{\frak a,\tau}$ and
$\mathscr U^o_{\frak a,\tau}$ the operators given by the formulas
\begin{equation}\label{eq.238}
               \mathscr U^e_{\mathfrak a,\tau}:=\sum_{j=0}^\infty\Psi_\tau(\lambda_j)\alpha_j\Psi^*_\tau(\lambda_j), \qquad
               \mathscr U^o_{\mathfrak a,\tau}:=\sum_{j=1}^\infty\Phi_\tau(\lambda_j)\alpha_j\Phi^*_\tau(\lambda_j).
\end{equation}

\begin{remark}\label{rm.213}
For every $\tau\in\Re\mathbf L_2$ the operators $\mathscr
U^e_{\frak a,\tau}$ and $\mathscr U^o_{\frak a,\tau}$ are nonnegative. In addition, in view
of Theorem~\ref{th.24} for every $\tau\in \Re\mathbf L_2$ we get
the equality
\begin{equation}\label{eq.250}
    \mathscr U^e_{\frak a_{\tau},\tau}=\mathcal I=\mathscr U^o_{\frak a_{\tau},\tau}.
\end{equation}
 \end{remark}

\begin{lemma}\label{le.213}
Assume that $\tau\in \mathbf L _2$ and that $\mathfrak
a=((\lambda_j,\alpha_j))_{j\in\mathbb Z_+}$ satisfies
condition~$(A_1)$. Then the following equivalences  hold:
\begin{equation}\label{eq.239}
            (\mathscr U^e_{\mathfrak a,\tau}>0)\Longleftrightarrow (A_3), \qquad \qquad
            (\mathscr U^o_{\mathfrak a,\tau}>0) \Longleftrightarrow (A_4).
\end{equation}
\end{lemma}

\begin{proof}  Taking into account the relations~\eref{eq.28}, we obtain that
\begin{equation}\label{eq.242}\fl
             \mathscr U^e_{\mathfrak a,\tau}
                   =(\mathcal I+\mathscr K_{\tau,N})\mathscr U^e_{\mathfrak a,0}(\mathcal I+\mathscr K_{\tau,N}^*),\qquad
              \mathscr U^o_{\mathfrak a,\tau}
                   =(\mathcal I+\mathscr K_{\tau,D})\mathscr U^o_{\mathfrak a,0}(\mathcal I+\mathscr K_{\tau,D}^*).
\end{equation}
Since the operators~$\mathcal I+\mathscr K_{\tau,N}$ and~$\mathcal I+\mathscr K_{\tau,D}$ are homeomorphisms of the space~$L_2$, it is enough to prove equivalences~\eref{eq.239} only for
$\tau=0$. Set
\[\fl
        \mathcal X=\{d\cos \lambda_{j} x\mid j\in\mathbb{Z_+},  d\in\Ran\alpha_{j}\}, \qquad
        \mathcal Y=\{d\sin \lambda_j x\mid j\in\mathbb{N}, d\in \Ran\alpha_j\}
\]
and observe that conditions $(A_3)$ and $(A_4)$ are equivalent to
completeness of the sets $\mathcal X$ and $\mathcal Y$
respectively. Since, in view of~\eref{eq.27},
\[
    \Psi^*_0(\lambda) f = \sqrt2\int_0^1 (\cos\lambda x) f(x)\,\rmd x,
    \qquad
    \Phi^*_0(\lambda) f = \sqrt2\int_0^1 (\sin\lambda x) f(x)\,\rmd x,
\]
we conclude that
\[\fl
          \ker\mathscr U^e_{\mathfrak a,0}=\bigcap_{j\in\mathbb Z_+}\ker\alpha_j\Psi^*_0(\lambda_j)=\mathcal X^\perp, \qquad
          \ker\mathscr U^o_{\mathfrak a,0}=\bigcap_{j\in\mathbb N}\ker\alpha_j\Phi^*_0(\lambda_j)=\mathcal Y^\perp.
\]
This justifies the equivalences of~\eref{eq.239}.
\end{proof}

\begin{proofof}{Proof of Theorem~\ref{th.29}}
Let $\tau\in \Re\mathbf L _2$ and $\mathfrak
a_\tau=((\lambda_j(\tau),\alpha_j(\tau)))_{j\in\mathbb Z_+}$.
Then~$(A_1)$ holds by Theorem~\ref{th.25}, while $(A_3)$ and
$(A_4)$ are satisfied in view of~\eref{eq.250} and Lemma~\ref{le.213}.
Next, in virtue of Lemma~\ref{le.212},
\begin{equation*}
          \sum_{n=0}^\infty\|P_{n,0}-\sum_{\lambda_j\in\Delta_n} P_{j,\tau}\|^2 <\infty.
\end{equation*}
Applying now Lemma~\ref{le.b1}, we conclude that there exists
$N_0\in\mathbb N$ such that
\begin{equation}\label{eq.245}
          \sum_{n=0}^N\sum_{\lambda_j\in\Delta_n} \rank P_{j,\tau}=\sum_{n=0}^N\rank P_{n,0}, \qquad N\ge N_0.
\end{equation}
It follows from~\eref{eq.29} and \eref{eq.216} that $\rank
P_{j,\tau}=\rank\alpha_j$ and $\rank P_{n,0}=r$ for
all~$j,n\in\mathbb Z_+$. Also, $\rank \alpha_0=r$ as $\alpha_0>0$.
Together with~\eref{eq.245} this justifies~$(A_2)$. The proof is
complete.
\end{proofof}

\section{\textbf{The Krein accelerant}}\label{sec.3}

\subsection{Proof of Theorem~\ref{th.12} }\label{subsec.31}

Let $H$ be an even function belonging to
$L_2((-1,1),M_r)$. We denote by~$\mathscr H$ an operator in $L_2$
given by
\begin{equation}\label{eq.30}
  (\mathscr H f)(x):=\int_0^1 H(x-t)f(t)\,\rmd t.
\end{equation}
Set $\mathscr H^a:=\chi_a\mathscr H \chi_a$, $a\in [0,1]$, where
$\chi_a$ is an operator in $L_2$ of multiplication by the
indicator of the interval~$(0,a]$, i.e.,
\[
        (\chi_a f)(x)=\left\{ \begin{array}{ll}
                                f(x), & \hbox{if $x\in (0,a]$,} \\
                                0,    & \hbox{if $x\in (a,1)$.}
                              \end{array}\right.
\]

\begin{remark}\label{re.31}
Taking into account Definition~\ref{de.11}, it is easy to see
that the following  equivalences take place:
\begin{equation}\label{eq.31}
        \eqalign{
            H\in\mathfrak H_2 \Longleftrightarrow \forall a\in [0,1] \quad
                                     \ker(\mathcal I+\mathscr H^a)=\{0\}; \\
            H\in\mathfrak H_2\Longleftrightarrow H^*\in\mathfrak H_2;\\
            H\in\Re\mathfrak H_2 \Longleftrightarrow (\mathcal I+\mathscr H)>0.}
\end{equation}
Definition~\ref{de.11} also implies that the set $\mathfrak H_2$
is open and, therefore, the set~$\mathfrak H_2\cap
C^1([-1,1],M_r)$ is dense everywhere in $\mathfrak H_2$.
\end{remark}

Suppose that $H\in\mathfrak H_2$. Then for every $a\in [0,1]$ the
operator $\mathcal I +\mathscr H^a$ is invertible in the algebra~$\mathcal
B(L_2)$ of bounded linear operators acting in~$L_2$. Since the operator~$\mathscr H^a$ depends continuously
on~$a\in[0,1]$, the mapping $[0,1]\ni a\mapsto (\mathcal I
+\mathscr H^a)^{-1}\in\mathcal B(L_2)$ is continuous. Denote
by~$\Gamma_{a,H}$ the kernel of the integral operator $-(\mathcal I +
\mathscr H^a)^{-1} \mathscr H^a$. Since $\mathscr H$ belongs to
the class of the Hilbert--Schmidt operators, the mapping
\begin{equation}\label{eq.32}
            [0,1]\times \mathfrak H_2\ni (a,H)\mapsto \Gamma_{a,H}\in L_2((0,1)^2,M_r)
\end{equation}
is continuous.

It can be easily seen that, for all $(x,t)\in (0,1)^2$ and $H\in \mathfrak H_2$,
\begin{equation}\label{eq.33}
             [\Gamma_{a,H}(x,t)]^*=\Gamma_{a,H^*}(t,x)=\Gamma_{a,H^*}(x,t).
\end{equation}
For every $H\in \mathfrak H_2$ and every $(x,t)\in \Omega$ we put
\begin{equation}\label{eq.34}\fl
            \widetilde R_H(x,t)= \int_0^x H(x-y)H(y-t)\,\rmd y + \int_0^x\int_0^x
H(x-u)\Gamma_{x,H}(u,v)H(v-t)\,\rmd v\,\rmd u.
\end{equation}
Since the functions~\eref{eq.32} are continuous and the shift
$t\mapsto H(\cdot-t)$ acts continuously in $L_2$-topology, one
concludes that $\widetilde R_H\in C(\Omega,M_r)$ and that the
mapping
\begin{equation}\label{eq.35}
             \mathfrak H_2\ni H\mapsto \widetilde R_H\in C(\Omega,M_r)
\end{equation}
is continuous.

\begin{lemma}\label{le.32} Let $H\in \mathfrak H_2$. Then the
function
  \begin{equation}\label{eq.36}
          R_H(x,t):=\left\{
                       \begin{array}{ll}
                         \widetilde R_H(x,t) -H(x-t), & \hbox{$0\le t \le x\le 1$,} \\
                         \qquad\qquad 0, & \hbox{$0\le x<t\le 1$,}
                       \end{array}
                     \right.
  \end{equation}
belongs to $G^+_2$ and equation~\eref{eq.19} has a unique
solution in the class $L_2(\Omega,M_r)$. Moreover,
the mapping $\mathfrak H_2\ni H\mapsto R_H\in G_2^+$ is continuous.
\end{lemma}

\begin{proof} That $R_H$ belongs to~$ G^+_2$ and the
mapping $H\mapsto R_H$ is continuous easily follows from continuity of the
mapping~\eref{eq.35}. Taking into account the definition of
$\Gamma_{a,H}$, we obtain that
\[
    H(x-t)+\Gamma_{x,H}(x,t)+\int_0^x \Gamma_{x,H}(x,\xi)H(\xi-t)\,\rmd \xi =0,  \qquad
    (x,t)\in \Omega.
\]
Straightforward calculations give
\begin{equation}\label{eq.19kr}
\Gamma_{x,H}(x,t)= \widetilde R_H(x,t) -H(x-t)= R_H(x,t),  \qquad
    (x,t)\in \Omega.
\end{equation}
Thus~$R_H$ is a solution of equation~\eref{eq.19}. Uniqueness
follows from Lemma~\ref{le.a2}.
\end{proof}

The proposition below follows from the results of~\cite{Kr56} (see also~\cite{GGK1}).

\begin{proposition}\label{pr.33}
Suppose that $H\in \mathfrak H_2\cap C([-1,1],M_r)$. Then, for
every $x,t\in [0,a]$,
\begin{equation}\label{eq.37}
            \frac{\partial \Gamma_{a,H}}{\partial a}(x,t)=\Gamma_{a,H}(x,a)\Gamma_{a,H}(a,t), \quad
            \Gamma_{a,H}(x,t)=\Gamma_{a,H}(a-x,a-t).
\end{equation}
\end{proposition}

Now we study the properties of the mapping~$\Theta$ defined by~\eref{eq.110}.

\begin{lemma}\label{le.34}
  \begin{enumerate}
    \item The mapping $\Theta: \mathfrak H_2\to \mathbf L_2$ is continuous and $\Theta(H^*)=[\Theta(H)]^*$ for
        every $H\in\mathfrak H_2$;
    \item  if $H\in\mathfrak H_2\cap C^1([-1,1],M_r)$, then $R_H\in C^1(\Omega,M_r)$ and
       \begin{equation}\label{eq.38}
              \frac{\partial}{\partial x}[R_H(x,x-t)]- R_H(x,0)R_H(x,t)=0, \qquad
              (x,t)\in\Omega.
       \end{equation}
  \end{enumerate}
\end{lemma}

\begin{proof} \emph{(i)}
Continuity of the mapping $\Theta$ easily follows
from definition and the continuity of the mapping $H\mapsto R_H$
(see Lemma~\ref{le.32}). By virtue of~\eref{eq.19kr} and the relation
\[
    \Gamma_{x,H}(x,\xi)=-H(x-\xi)-\int_0^x H(x-u)\Gamma_{x,H}(u,\xi)\,\rmd u,  \qquad
    (x,\xi)\in \Omega,
\]
we obtain that
\begin{equation*}\fl
       \Theta(H)(x)  = H(x)
                     - \int_0^x H(x-\xi)H(\xi)\,\rmd \xi
                     - \int_0^x\int_0^x H(x-u)\Gamma_{x,H}(u,v) H(v)\,\rmd u\,\rmd v .
\end{equation*}
Now assume that, in addition, $H\in C([-1,1],M_r)$;  using then~\eref{eq.33} and the second equality in~\eref{eq.37}, we get that
\begin{eqnarray*}\fl
    [\Theta(H)]^*(x) = H^*(x) - \int_0^x H^*(\xi)H^*(x-\xi)\,\rmd\xi \\
        \qquad\qquad - \int_0^x\int_0^x H^*(v)\Gamma_{x,H^*}(v,u)H^*(x-u)\,\rmd u\,\rmd v \\
    =H^*(x)-\int_0^x H^*(x-y)H^*(y)\,\rmd y \\
        \qquad\qquad - \int_0^x\int_0^x H^*(x-y)\Gamma_{x,H^*}(y,t)H^*(t) \,\rmd y\,\rmd t \\
    =\Theta(H^*)(x)
\end{eqnarray*}
as claimed. Since the set $\mathfrak H_2\cap C([-1,1],M_r)$ is dense in~$\mathfrak H_2$ and $\Theta$ is continuous, the equality~$\Theta(H^*)=[\Theta(H)]^*$ holds for all $H \in \mathfrak H_2$.

\emph{(ii)} Let $H\in\mathfrak H_2\cap C^1([-1,1],M_r)$ and $R=R_H$, so that
\begin{equation}\label{eq.39}
    R(x,t)+H(x-t)+\int_0^xR(x,\xi)H(\xi-t)\,\rmd \xi=0,
        \qquad (x,t)\in\Omega.
\end{equation}
It follows from Proposition~\ref{pr.33}  that the function
$a\mapsto\Gamma_{a,H}(u,v)$ is continuously differentiable for
$a\ge\max\{u,v\}$. Therefore, taking into account~\eref{eq.37},
\eref{eq.34}, and~\eref{eq.36}, we obtain $R\in C^1(\Omega,M_r).$
Since~$R$ is a solution of the equation~\eref{eq.19} and the
function~$H$ is even, one derives the equality
\begin{equation}\label{eq.310}
    R(x,x-t)+H(t)+\int_0^xR(x,x-\xi)H(\xi-t)\,\rmd\xi=0, \quad
        (x,t)\in\Omega.
\end{equation}
Now differentiate~\eref{eq.310} in the variable $x$ to get
\begin{equation}\label{eq.311}
    \frac{\partial}{\partial x}[R(x,x-t)]
        +R(x,0)H(x-t)
        +\int_0^x\frac{\partial}{\partial x}[R(x,x-\xi)]H(\xi-t)\,\rmd\xi=0.
\end{equation}
Multiplying both sides of equality~\eref{eq.39} by~$R(x,0)$ from
the left, we obtain
\begin{equation*}
    R(x,0)R(x,t)+R(x,0)H(x-t)+\int_0^xR(x,0)R(x,\xi)H(\xi-t) \,\rmd\xi=0.
\end{equation*}
Subtracting~\eref{eq.311} from the above equality, we
find that for the function
\[
    F(x,t):=\frac{\partial}{\partial x}[R(x,x-t)]-R(x,0)R(x,t)
\]
the following relation holds for all~$(x,t)\in\Omega$:
\begin{equation*}
    F(x,t)+\int_0^x F(x,\xi)H(\xi-t)\,\rmd \xi=0.
\end{equation*}
Lemma~\ref{le.a2} now yields $F\equiv0$, and the proof is
complete.
\end{proof}

\begin{proofof}{Proof of Theorem~\ref{th.12}}
Assume that $H\in\mathfrak H_2$ and set
\begin{eqnarray*}
    Q_{o,H}(x,t):=\frac 12\left[R_H\Bigl(x,\frac{x+t}{2}\Bigr)
            -R_H\Bigl(x,\frac{x-t}{2}\Bigr)\right],  \\
    Q_{e,H}(x,t):=\frac 12\left[R_H\Bigl(x,\frac{x+t}{2}\Bigl)
            + R_H\Bigl(x,\frac{x-t}{2}\Bigl)\right],
\end{eqnarray*}
where $(x,t)\in\Omega$. We extend  the functions $Q_{o,H}$ and $Q_{e,H}$ onto the square~$(0,1)^2$ by setting $Q_{o,H}\equiv 0$
and $Q_{e,H}\equiv 0$ in~$\Omega^-=[0,1]^2\setminus \Omega$.
It follows from~\eref{eq.36} and the continuity of the mapping~\eref{eq.35} that $Q_{o,H}$ and $Q_{e,H}$ belong
to~$G_2$ and that the mappings
\begin{equation}\label{eq.314}
    \mathfrak H_2 \ni H\mapsto Q_{o,H}\in G_2, \qquad
    \mathfrak H_2 \ni H\mapsto Q_{e,H}\in G_2
\end{equation}
are continuous. For $x\in [0,1]$ and $\lambda\in\mathbb C$, we consider the functions
\begin{equation*}
    \eqalign{ \varphi(x,\lambda):=\sin \lambda x I
                +\int_{0}^{x}(\sin\lambda t)\,Q_{o,H}(x,t)\, \rmd t, \\
             \psi(x,\lambda):=\cos \lambda x I
                +\int_{0}^{x}(\cos\lambda t)\,Q_{e,H}(x,t)\, \rmd t.}
\end{equation*}
Straightforward transformations give
\begin{equation}\label{eq.316}
\eqalign{ \varphi(x,\lambda)=\sin \lambda x I
        + \int_{0}^{x}(\sin\lambda(x-2s))\,R(x,x-s)\, \rmd s, \\
         \psi(x,\lambda)=\cos \lambda x I
        + \int_{0}^{x}(\cos\lambda(x-2s))\,R(x,x-s)\, \rmd s.}
\end{equation}

Suppose that $H\in \mathfrak H_2\cap C^1([-1,1],M_r)$. According to Lemma~\ref{le.34}, the function $R_H$ belongs to $C^1(\Omega,M_r)$ and
equality~\eref{eq.38} holds. Using~\eref{eq.316} and~\eref{eq.38}, we arrive at the relations
\begin{equation*}
    \left(\frac{\rmd}{\rmd x}-\tau\right)\varphi(x,\lambda)
        = \lambda\psi(x,\lambda),
        \qquad
    \left(\frac{\rmd}{\rmd x}+\tau\right)\psi(x,\lambda)
        =-\lambda \varphi(x,\lambda),
  \end{equation*}
where $\tau =\Theta(H)$. Since $\varphi(0,\lambda)=0$ and $\psi(0,\lambda)=I$,
the functions $\varphi$ and $\psi$ are solutions of problems~\eref{eq.11} and~\eref{eq.12}. In view of Theorem~\ref{th.21} this establishes the relations
\begin{equation}\label{eq.317}
    Q_{o,H}=K_{\tau,D},     \qquad
    Q_{e,H}=K_{\tau,N},     \qquad
    \tau=\Theta(H)
\end{equation}
for all $H\in \mathfrak H_2\cap C^1([-1,1],M_r)$.
Since the set $\mathfrak H_2\cap C^1([-1,1],M_r)$ is dense everywhere in
$\mathfrak H_2$ and the mappings~$\Theta$,
\eref{eq.314},
 $\mathbf L_2\ni\tau\to\,K_{\tau,D}\in\, G_2^+$, and
 $\mathbf L_2\ni\,\,\tau\to K_{\tau,N}\in G_2^+ $
are continuous (see Theorem~\ref{th.21}), it follows that equalities~\eref{eq.317} hold for all $H\in \mathfrak H_2$. The proof is complete.
\end{proofof}


\subsection{Proof of Theorem~\ref{th.13}.}\label{subsec.32}

Consider the integral equation
\begin{equation}\label{eq.318}
           P(x,x-s)+\tau(s)+\int_s^x\tau(u)P(u,s)\,\rmd u=0,\qquad (x,s)\in\Omega,
\end{equation}
where  $\tau$ is a known function from $\mathbf L_2$, and $P$ is
an unknown function from the class~$G_2^+$.

\begin{lemma}\label{le.35}
The equation~\eref{eq.318} has at most one solution. If
$H\in \mathfrak H_2\cap C^1([-1,1],M_r)$ and $\tau=\Theta(H)$, then the solution~$R_H$ of the Krein equation~\eref{eq.19} is also a
solution of the equation~\eref{eq.318}.
\end{lemma}

\begin{proof} To prove uniqueness, it is enough to show that the corresponding homogeneous equation
\begin{equation}\label{eq.319}
      F(x,x-s) =-\int_s^x \tau(u) F(u,s)\,\rmd u,\qquad (x,s)\in \Omega,
\end{equation}
has only zero solution in  $G_2^+$. Let a function $F\in G_2^+$
be a solution of~\eref{eq.319} and set
\[
    g(x):=\int_0^x\|F(x,s)\|\,\rmd s, \qquad x\in [0,1].
\]
Then $g$ is continuous on~$[0,1]$ and, in view of~\eref{eq.319},
\begin{equation*}\fl
    g(x) = \int_0^x\|F(x,x-s)\|\,\rmd s
        \le\int_0^x\|\tau(u)\|\int_0^{u} \|F(u,s)\|\, \rmd s\,\rmd u
        \le \int_0^x\|\tau(u)\|g(u)\,\rmd u.
\end{equation*}
Now Gronwall's inequality implies that $g\equiv 0$, whence~$F=0$.

Let $H\in \mathfrak H_2$ and $\tau=\Theta(H)$. Since
$\tau=\Theta(H)=-R_H(\cdot,0)$,  we get by virtue of
equality~\eref{eq.38} that
\[
    R_H(x,x-t) +\tau(t) +\int\nolimits_t^x\tau(\xi)R_H(\xi,t)\,\rmd \xi=0, \qquad (x,t)\in \Omega.
\]
Thus $R_H$ is a solution of equation~\eref{eq.318}.
\end{proof}

\begin{lemma}\label{le.36}
For every $\tau\in \mathbf L_2$ equation~\eref{eq.318} has a
unique solution~$P_{\tau}$ in~$ G_2^+$ and the mapping $\mathbf
L_2\ni\tau\mapsto P_{\tau}\in G_2^+$ is continuous. Moreover, if
$\tau\in C([0,1],M_r)$, then $P_\tau$ is continuous in $\Omega$,
the function $(x,t)\mapsto P_{\tau}(x,x-t)$ has continuous partial
derivative in the variable~$x$, and
\begin{equation}\label{eq.320}
       \frac{\partial}{\partial x}[P_{\tau}(x,x-t)]=-\tau(x)P_{\tau}(x,t),\qquad (x,t)\in\Omega.
\end{equation}
\end{lemma}

\begin{proof}
For every $n\in\mathbb N$ and $(x,s)\in \Omega$, we set
\begin{equation}\label{eq.321}
            P_{\tau,1}(x,x-s) := \tau(s), \qquad
            P_{\tau,n+1}(x,x-s): = \int_s^x \tau(u) P_{\tau,n}(u,s)\,\rmd u.
\end{equation}
It follows by induction that the recurrence relation~\eref{eq.321} yields the equality
\begin{equation}\label{eq.322}
         P_{\tau,n+1}(x,s)=\int_{\Pi^*_{n}(x,s)}\tau(x-s+\xi_{n}(y))\tau(y_1)\cdots \tau(y_n)\,\rmd y_1\dots\,\rmd y_n,
\end{equation}
where, for $n\in\mathbb N$ and $(x,s)\in \Omega$, we set
$\xi_n(y):=\sum_{j=1}^n (-1)^{j+1} y_j$ and
\[\fl
    \Pi^*_{n}(x,s):=  \{y=(y_1,\dots,y_{n})\in \mathbb R^{n} \mid
        0 \le y_{n} \le y_{n-1} \le \dots \le y_1
            \le  x-s+ \xi_{n}(y) \le x\}.
\]
We extend the functions $P_{\tau,n}$, $n\in \mathbb N$, onto the
square~$[0,1]\times [0,1]$ by setting~$P_{\tau,n}\equiv 0$ in
$\Omega^-$. Since the function~$\tau$ belongs to $\mathbf L_2$,
for any  $n\in \mathbb N$ the functions $P_{\tau,n+1}$ are
continuous in the square $[0,1]^2$. Taking into
account~\eref{eq.321} and using the Cauchy--Bunyakowski inequality
and Fubini's theorem, we find that for every $n\in \mathbb N$ and
$x\in [0,1]$ it holds
\begin{eqnarray*}\fl
     \|P_{\tau,n+1}(x,\,\cdot\,)\|^2_{\mathbf L_2}
      =\int_0^1 ||P_{\tau,n+1}(x,s)||^2\,\rmd s \\
        \le \frac {1}{n!} \int_0^1 \int_{\Pi^*_{n}(x,s)} ||\tau(x-s+\xi_{n-1}({\mathrm y}))||^2\,
                ||\tau(y_1)||^2 \cdots ||\tau(y_{n})||^2\,
                    \rmd y_1 \dots \rmd y_{n}\rmd s\\
      \le  \frac {1}{n!} \int_{\Pi_{n+1}}||\tau(t_1)||^2\cdots ||\tau(t_{n+1})||^2 \,\rmd t_1\dots\,\rmd t_{n+1}\\
      = \frac 1{(n!) (n+1)!} \Bigl(\int_0^1\|\tau(\xi)\|^{2}\,\rmd \xi\Bigr)^{n+1},
\end{eqnarray*}
where $\Pi_n = \{t:= (t_1,\dots,t_n)\in \mathbb{R}^n \mid 0 \le
t_n \le \dots \le t_1\le 1\}$. Therefore
\[
    \|P_{\tau,n+1}(x,\cdot)\|_{\mathbf L_2}\le (n!)^{-1} \|\tau\|_{\mathbf L_2}^{n+1}.
\]
Similarly we obtain that
\[
    \|P_{\tau,n+1}(\,\cdot\,,x)\|_{\mathbf L_2}
        \le (n!)^{-1}\|\tau\|_{\mathbf L_2}^{n+1}.
\]
Now, taking into account continuity of the functions
$P_{\tau,n+1}$, we conclude that $P_{\tau,n+1}$  belongs to
$G_2^+$, and, moreover,
\begin{equation}\label{eq.323}
    \|P_{\tau,n+1}\|_{G_2}
        \le  (n!)^{-1}\|\tau\|_{\mathbf L_2}^{n+1}, \qquad n\in\mathbb N.
\end{equation}
Taking into account~\eref{eq.321} and~\eref{eq.323}, we see
that the function
\begin{equation}\label{eq.324}
          P_{\tau}:=\sum\limits_{n=1}^{\infty}(-1)^nP_{\tau,n}
\end{equation}
is a solution of equation~\eref{eq.318}. Uniqueness of
solution of equation~\eref{eq.318} follows from Lemma~\ref{le.35}.

Next we prove continuous dependence of
$P_\tau$ on~$\tau$. It follows from~\eref{eq.323} and~\eref{eq.324} that it is sufficient to prove continuity of the mappings
\begin{equation}\label{eq.325}
            \mathbf L_2\ni\tau\mapsto P_{\tau,n}\in G_2^+, \qquad n\in\mathbb N.
\end{equation}
Let $\tau_1,\tau_2\in\mathbf L_2$. Using~\eref{eq.322} and the
inequality
\[
    \bigl|\prod_{k=1}^n a_k - \prod_{k=1}^n b_k\bigr|^2
               \le n\sum_{k=1}^n |a_k-b_k|^2
                \prod_{j\ne k} (|a_j| + |b_j|)^2,
\]
valid for all $(a_k)_{k=1}^n$ and $(b_k)_{k=1}^n$ in~$\mathbb C^n$,
and reasoning similarly as when deriving the estimate~\eref{eq.323}, we obtain
the estimate
\begin{equation*}
    \|P_{\tau_1,n+1} -  P_{\tau_2,n+1}\|^2_{G_2^+}
            \le \Bigl(\frac{n+1}{n!}\Bigr)^2\|\tau_1-\tau_2\|^2_{\mathbf L_2}
            \Bigl(\|\tau_1\|_{\mathbf L_2}+ \|\tau_2\|_{\mathbf L_2} \Bigr)^{2n},
\end{equation*}
which yields continuity of the mappings~\eref{eq.325}.

Finally, assume that~$\tau\in C([0,1],M_r)$. Since
the function $P_\tau$ verifies~\eref{eq.318}, $P_\tau$ is continuous, and the function~$(x,s)\mapsto P_\tau(x,x-s)$ is absolutely continuous in~$x$. Differentiation of~\eref{eq.318} in the variable~$x$ leads to~\eref{eq.320}. The proof is complete.
\end{proof}

\begin{proofof}{Proof of Theorem~\ref{th.13}}
For any $\tau\in\mathbf L_2$ we denote by $\mathscr P_{\tau}$ the
integral operator with kernel $P_{\tau}$. It follows from
Lemma~\ref{le.36} that $\mathscr P_{\tau}$ belongs to $\mathfrak
G_2^+$ and that the mapping $\mathbf L_2\ni\tau\mapsto \mathscr
P_\tau\in\mathfrak G_2^+$ is continuous. Hence in view of
Lemma~\ref{le.a3} we obtain that the mapping $\Upsilon$ that acts
from $\mathbf L_2$ to $L_2((-1,1),M_r)$ via
\begin{equation*}
    \left[\Upsilon(\tau)\right](x)=\left[(\mathcal I+\mathscr P_{\tau})^{-1}\tau\right](|x|), \qquad
     x\in (-1,1),
\end{equation*}
is also continuous. Next we show that the mapping $\Upsilon$ is the left inverse of~$\Theta$, i.e., that the relation
\begin{equation}\label{eq.327}
    \Upsilon(\Theta(H))=H
\end{equation}
holds for all~$H\in\mathfrak H_2$. Since the set $\mathfrak H_2\cap C^1([-1,1], M_r)$ is dense everywhere in
$\mathfrak H_2$ and the mappings~$\Theta$ and~$\Upsilon$ are
continuous, it is enough to prove equality~\eref{eq.327} for
$ H\in\mathfrak H_2\cap C^1([-1,1], M_r)$. Assume $H$ is such and set~$\tau=\Theta(H)$. Then it follows from Lemmas~\ref{le.35} and~\ref{le.36} that $P_\tau=R_H$. Therefore, recalling~\eref{eq.110}, we obtain that
\[
    \tau(x)=[\Theta(H)](x)= H(x) +\int\nolimits_0^x P_\tau(x,\xi)
        H(\xi) \,\rmd \xi, \qquad x\in (0,1),
\]
and, therefore, $\Upsilon(\Theta(H))=\Upsilon(\tau)=H$.

It remains to show that~$\Upsilon$ is also the right inverse of~$\Theta$, i.e., that, for
every $\tau\in \mathbf L_2 $,
\begin{equation}\label{eq.328}
 \Upsilon(\tau)\in\mathfrak H_2 \qquad \rm{and} \qquad
  \Theta(\Upsilon(\tau))=\tau.
\end{equation}
Continuity of the mappings  $\Theta$ and $\Upsilon$ implies that it is
enough to prove~\eref{eq.328} for continuous~$\tau$.
In this case for $H:=\Upsilon(\tau)$ we get the relation
\begin{equation}\label{eq.329}
    \tau(x)=H(x)+\int\nolimits_0^x P_\tau(x,\xi) H(\xi)\, \rmd\xi, \qquad x\in [0,1].
\end{equation}
According to Lemma~\ref{le.36}, the function $P_\tau$ is
continuous in $\Omega$, and, therefore, the function~$H$ is continuous as well. Consider a function~$F$ defined for $(x,t)\in\Omega$ by
\[
       F(x,t):=P_{\tau}(x,x-t)+H(t)+\int_0^xP_{\tau}(x,x-\xi)H(\xi-t)\,\rmd\xi.
\]
Then the function~$F$ is continuous in $\Omega$, has a continuous partial derivative in the variable~$x$ by virtue of Lemma~\ref{le.36}, and
\begin{equation*}
        \frac{\partial}{\partial x}[F(x,x-t)]=-\tau(x)F(x,t),\qquad
            (x,t)\in\Omega.
\end{equation*}
Since
\begin{equation*}
    F(t,0)=-\tau(t)+H(t)+\int_0^tP_{\tau}(t,\xi)H(\xi)\,\rmd \xi=0,
        \qquad t\in [0,1],
\end{equation*}
we obtain that
\[
    F(x,x-t)=-\int_t^x\tau(\xi)F(\xi,t)\, \rmd\xi,\qquad t\le x\le 1.
\]
However, as it was shown in the proof of Lemma~\ref{le.35}, the
last equality is only possible for $F\equiv 0$. Therefore $P_\tau$
is a solution of the Krein equation~\eref{eq.19} with
$H=\Upsilon(\tau)$. Now Proposition~\ref{th.a1a} implies that~$H$ is an accelerant; moreover, $\Theta(\Upsilon(\tau))=\tau$ in view of~\eref{eq.329}. Finally, the equality~$\Theta(H^*)=[\Theta(H)]^*$ follows from Lemma~\ref{le.34}~\emph{(i)}. The proof is complete.
\end{proofof}

\subsection{The operators $\mathscr H_o$ and $\mathscr H_e$} \label{subsec.33}

For any $H\in L_2((-1,1),M_r)$  we denote by $\mathscr H_o$ and
$\mathscr H_e$ integral operators  that act in $L_2$ by the
formulas
\begin{equation}\label{eq.330}\fl
    (\mathscr H_\mathrm{o} f)(x):=\int_0^1 H_o(x,t)f(t)\,\rmd t, \qquad
    (\mathscr H_e f)(x):=\int_0^1 H_e(x,t)f(t)\,\rmd t,
\end{equation}
where
\begin{equation*}
\eqalign{
H_{o}(x,t)&:=\frac{1}{2}\left[H\Bigr(\frac{x-t}{2}\Bigr)
        - H\Bigr(\frac{x+t}{2}\Bigr) \right], \\
H_{e}(x,t)&:=\frac{1}{2}\left[H\Bigr(\frac{x-t}{2}\Bigr)
        + H\Bigr(\frac{x+t}{2}\Bigr) \right].}
\end{equation*}

\begin{proposition}\label{pr.37}
Let $H\in\mathfrak H_2$ and $\tau=\Theta(H)$. Then
\begin{equation}\label{eq.332}\fl
          (\mathcal I +\mathscr K_{\tau,N})(\mathcal I +\mathscr H_e)
                    (\mathcal I +\mathscr K_{\tau^*,N}^*) =\mathcal I
          =(\mathcal I+\mathscr K_{\tau,D})(\mathcal I +\mathscr H_o) (\mathcal I+\mathscr K_{\tau^*,D}^*).
\end{equation}
\end{proposition}

\begin{proof}
We shall prove the first
equality in~\eref{eq.332}; the second one is proved
similarly. It follows from the Krein equation~\eref{eq.19} that
\begin{eqnarray*}
    R_H\Bigr(x,\frac{x-t}{2}\Bigr)+H\Bigr(\frac{x+t}{2}\Bigr)
        &+\int_0^x R_H(x,\xi)H\Bigr(\xi-\frac{x-t}{2}\Bigr)\,\rmd\xi=0, \\
    R_H\Bigr(x,\frac{x+t}{2}\Bigr)+H\Bigr(\frac{x-t}{2}\Bigr)
        &+\int_0^x R_H(x,\xi)H\Bigr(\xi-\frac{x+t}{2}\Bigr)\,\rmd\xi=0
\end{eqnarray*}
for $(x,t)\in\Omega$. Combining these relations and taking into account Theorem~\ref{th.12} proved above, we arrive at the equality
\begin{equation}\label{eq.333}
    K_{\tau,N}(x,t)+H_e(x,t)+\int_0^xK_{\tau,N}(x,\xi)H_e(\xi,t)\,\rmd\xi=0
\end{equation}
for all $(x,t)\in\Omega$. By virtue of Lemma~\ref{le.34},
$\Theta(H^*)=[\Theta(H)]^*=\tau^*$. Thus, according to~\eref{eq.333}, we find that for $(x,t)\in\Omega$ it holds
\begin{equation}\label{eq.334}
    K_{\tau^*,N}(x,t)+H^*_e(x,t)+\int_0^xK_{\tau^*,N}(x,\xi)H^*_e(\xi,t)\,\rmd \xi=0.
\end{equation}
Equalities~\eref{eq.333}, \eref{eq.334}, and Theorem~\ref{th.a1}
now lead to the relation
\[
    \mathcal I +\mathscr H_e =(\mathcal I +\mathscr K_{\tau,N})^{-1}
        (\mathcal I +\mathscr K_{\tau^*,N}^*)^{-1},
\]
which yields the first equality in~\eref{eq.332}. The lemma is proved.
\end{proof}

\begin{proposition}\label{pr.38}
If  $\ker(\mathcal I +\mathscr H_e)=\{0\}$, then  $\mathcal
I+\mathscr H_e$ is a bijection of~$W_2^1$.
\end{proposition}

\begin{proof} It is enough to prove that $\mathscr H_e$ maps $W_2^1$ to itself and that the operator $\mathscr H_e: W_2^1\to W_2^1$ is compact.
If $H\in C^1([-1,1],M_r)$, then integration by parts gives
\begin{equation}\label{eq.335}
    (\mathscr H_e f)'= \mathscr H_o f' - H_o(\cdot,1)f(1),    \qquad f\in W_2^1.
\end{equation}
Using the fact that the operator of differentiation considered on the domain~$W_2^1$ is closed in~$L_2$ and that the space~$C^1([-1,1],M_r)$ is dense in~$L_2((-1,1),M_r)$, we conclude that the equality~\eref{eq.335} holds also
for an arbitrary function $H\in L_2((-1,1),M_r)$. The equality~\eref{eq.335}
implies that the operator $\mathscr H_e$ maps the space $W_2^1$ to
itself; moreover,
\begin{equation}\label{eq.345}
                 \|\mathscr H_e\|_{W_2^1\to W_2^1} \le C\|H\|_{L_2((-1,1),M_r)},
\end{equation}
where $C$ is a constant independent of~$H$. If the function $H$ is a
trigonometric polynomial, then~$\mathscr H_e$ is an operator of finite rank. Approximating an arbitrary function~$H\in L_2((-1,1),M_r)$ by a sequence of
trigonometric polynomials and taking into account~\eref{eq.345},
we conclude that the operator~$\mathscr H_e: W_2^1\to W_2^1$ is compact.
\end{proof}

\section{The inverse spectral problem}\label{sec.4}

\subsection{Proof of Theorem~\ref{th.41}}\label{subsec.41}

We shall base the proof of Theorem~\ref{th.41} on the following two lemmas.

\begin{lemma}\label{le.42}
Let for a sequence $\mathfrak
a=((\lambda_j,\alpha_j))_{j\in\mathbb Z_+}$ condition $(A_1)$ hold
and $\nu=\nu^\mathfrak a$. Then the limit~\eref{eq.18} exists in
the topology of the space~$L_2((-1,1),M_r)$. Moreover, the
function $H=H_\nu$ is even and hermitian, and for the operators
$\mathscr H_e$ and $\mathscr H_o$ the following relations hold:
\begin{equation} \label{eq.43}
   \mathcal I +\mathscr H_e=\mathscr U^e_{\mathfrak a,0}, \qquad
   \mathcal I +\mathscr H_o=\mathscr U^o_{\mathfrak a,0}.
\end{equation}
\end{lemma}

\begin{proof}
In view of~\eref{eq.16} and~\eref{eq.17} we have to show that the series
\begin{equation} \label{eq.44}
    \sum_{n=1}^{\infty}\Bigl[\cos(2\pi nx)I
      - \sum_{\lambda_j\in\Delta_n}\cos(2\lambda_jx)\,\alpha_j\Bigr]
\end{equation}
converges in  $L_2((-1,1),M_r)$.  Let $\widetilde{\lambda}_j$ and $\beta_n$ be given by~\eref{eq.218}. Then
\[
    \cos 2\lambda_j x=\cos 2\pi nx
        - 2 x \widetilde{\lambda}_j\,\sin 2\pi nx -\vartheta_{n,j}(x),
         \qquad  n\in\mathbb N,\quad \lambda_j\in\Delta_n,
\]
where
\[
    \vartheta_{n,j}(x)
            :=(1-\cos 2x\widetilde{\lambda}_j)\cos 2\pi nx
            - (2\widetilde{\lambda}_jx - \sin 2\widetilde{\lambda}_j x)
                \sin 2\pi nx.
\]
This yields the equality
\[\fl
    \cos(2\pi nx)I - \sum_{\lambda_j\in\Delta_n}\cos(2\lambda_jx)\alpha_j
        = \cos(2\pi nx)\beta_n+2x\sin (2\pi nx)\gamma_n
            + \sum_{\lambda_j\in\Delta_n}\vartheta_{n,j}(x) \alpha_j,
\]
in which
$\gamma_{n}:=\sum\limits_{\lambda_j\in\Delta_n}\widetilde{\lambda}_j\alpha_j$.
According to condition $(A_1)$
\begin{equation}\label{eq.45m}\fl
    \sum_{n=1}^{\infty}\|\beta_n\|^2< \infty, \quad
       \sup\limits_{j\in\mathbb Z_+}\|\alpha_j\|<\infty,\quad
         \sum_{j=1}^{\infty} |\widetilde{\lambda}_j|^2< \infty, \quad
             \sup\limits_{n\in\mathbb N} \sum\limits_{\lambda_j\in\Delta_n} 1 <\infty,
\end{equation}
which, in particular, implies that~$\sum_{n=1}^{\infty}\|\gamma_n\|^2<\infty$. From the above we conclude that the series
 $\sum_{n=1}^\infty\cos(2\pi nx)\beta_n$
and
 $\sum_{n=1}^\infty\sin (2\pi nx)\gamma_n$
converge in the topology of the space~$L_2((-1,1),M_r)$. It is easy to see that
\[
  \max\limits_{|s|\le 1}|\vartheta_{n,j}(s)|\le 6|\widetilde{\lambda}_j|^2,
                                                   \qquad n\in\mathbb{N}, \qquad \lambda_j\in\Delta_n.
\]
Therefore,
\[
    \sum_{n=1}^{\infty} \sum_{\lambda_j\in\Delta_n}
            \|\vartheta_{n,j}\alpha_j \|_{L_2((-1,1),M_r)}
        \le 6\left(\sup\limits_{j\in\mathbb Z_+}\|\alpha_j\| \right)
            \sum_{j=1}^{\infty} |\widetilde{\lambda}_j|^2 <\infty .
\]
Combining the above, we conclude that the series~\eref{eq.44}
converges. The fact that the function~$H_\nu$ is even and
hermitian is obvious.

Now construct operators $\mathscr H_e$ and $\mathscr H_o$ via formulas~\eref{eq.330} taking there $H=H_\nu$. We shall prove the first equality in~\eref{eq.43}, since the second one is proved similarly. Let
\[
    H_{\nu,n}(x):= \int_0^{\pi(n+1/2)}
          2\cos(2\lambda x) \, \rmd(\nu-\nu_0)(\lambda), \qquad
x\in [-1,1], \qquad n\in\mathbb N.
\]
We denote by $\mathscr H_{e,n}$ the operator in~$L_2$ that acts via the formula
\[
    (\mathscr H_{e,n}f)(x)
        :=\frac{1}{2}\int_0^1 \Bigl[H_{\nu,n}\Bigl(\frac{x-t}{2}\Bigr)
            + H_{\nu,n}\Bigl(\frac{x+t}{2}\Bigr) \Bigr] f(t)\, \rmd t,
        \qquad x\in [-1,1].
\]
Since $ H_{\nu,n}\to H_\nu$ as $n\to \infty$, the sequence~$(\mathscr H_{e,n})_{n\in\mathbb N}$ converges in the operator norm to the
operator $\mathscr H_e$. It is easy to see that
\[
    \mathscr H_{e,n}
        =\sum\limits_{\lambda_j\le\pi (n+1/2)} \Psi_0(\lambda_j)\alpha_j\Psi^*_0(\lambda_j)
        -\sum\limits_{k=0}^n P_{k,0}.
\]
The required equality follows now from~\eref{eq.215} and~\eref{eq.238}.
\end{proof}

\begin{lemma}\label{le.43}
Let a sequence $\mathfrak a=((\lambda_j,\alpha_j))_{j\in\mathbb
Z_+}$ verifies condition~$(A_1)$. Then the following implications are
true:
      $(A_3) \Longrightarrow (A_4)$
and
      $ [(A_4)\wedge (\lambda_0=0, \alpha_0>0)]\Longrightarrow (A_3)$.
\end{lemma}

\begin{proof}
Let conditions $(A_1)$ and $(A_3)$ hold. Then it follows
from~\eref{eq.43} and~\eref{eq.239} that $\ker(\mathcal I+\mathscr
H_e)=\{0\}$. Hence, taking into account Proposition~\ref{pr.38},
we obtain that $(\mathcal I+\mathscr H_e) W_2^1=W_2^1$.  Therefore, for an
arbitrary $g\in L_2$ there exists $f\in W_2^1$ such that
$g=[(\mathcal I+\mathscr H_e)f]'$. Put
$c_j:=\alpha_j\Psi_0^*(\lambda_j)f$, $j\in\mathbb Z_+$. In view
of~\eref{eq.43} and~\eref{eq.238}
\begin{equation}\label{eq.46}
   (\mathcal I+\mathscr H_e)f=\sum\limits_{j=0}^\infty\Psi_0(\lambda_j)c_j.
\end{equation}
Taking into account the relation
\[\fl
    \lambda_j\Psi_0^*(\lambda_j)f=\sqrt2\int\limits_0^1(\lambda_j\cos\lambda_j t)\,f(t)\, \rmd t
        = (-1)^n\sqrt2(\sin\widetilde\lambda_j) f(1)
        -\Phi_0^*(\lambda_j)f',\qquad \lambda_j\in\Delta_n,
\]
the inequalities in~\eref{eq.45m}, and Lemma~\ref{le.211}, we conclude that
$\sum_{j=0}^\infty \|\lambda_j c_j\|^2_{\mathbb C^r}<\infty$. In view of Lemma~\ref{le.211} the series in~\eref{eq.46} is termwise differentiable and therefore
    $g=-\sum\limits_{j=1}^\infty\lambda_j\Phi_0(\lambda_j)c_j$.
Since $g$ is arbitrary, the system $\{d\sin \lambda_j x\mid
j\in\mathbb{N}, d\in \Ran\alpha_j\}$ is complete in the space
$L_2$, and the implication~$(A_3)\Longrightarrow (A_4)$ follows.

Now let the conditions $(A_1)$, $(A_4)$ hold and let $\lambda_0=0$
and $\alpha_0>0$. Then $\ker(\mathcal I+\mathscr H_o)=\{0\}$ by
virtue of~\eref{eq.43} and \eref{eq.239}. Since the operator
$\mathscr H_o:L_2\to L_2$ is compact, $\Ran(\mathcal I +\mathscr
H_o)=L_2$. Let $J: L_2\to L_2$ be an integration operator, i.e.,
\[
    (Jf)(x)=\int_0^x f(t)\, \rmd t.
\]
The range of the operator $ J$ is everywhere dense in~$L_2$.
Therefore the range of the operator $J(\mathcal I +\mathscr H_o)$
is also everywhere dense in~$L_2$. In view of~\eref{eq.43} for an
arbitrary $f\in L_2$ we get
\[
   J(\mathcal I+\mathscr H_o)f =\sum\limits_{j=1}^\infty J\Phi_0(\lambda_j)\alpha_j\Phi_0^*(\lambda_j)f
   =\sum\limits_{j=1}^\infty
\frac{\sqrt2}{\lambda_j}d_j-\sum\limits_{j=1}^\infty\frac1{\lambda_j}\Psi_0(\lambda_j)d_j,
\]
where $d_j:=\alpha_j\Phi_0^*(\lambda_j)f$. Since the matrix~$\alpha_0$ is nonsingular, the vector~$\sum_{j=1}^\infty \sqrt2d_j/{\lambda_j}$ is in the range of~$\alpha_0$. Therefore the closed
linear hull of the system
\[
    \{(\cos\lambda_j x)\,d \mid  j\in\mathbb Z_+,\  d\in \Ran\alpha_j\}
\]
 contains the range of~$J(\mathcal I +\mathscr H_o)$ and thus $(A_3)$ holds.
\end{proof}

\begin{proofof}{Proof of Theorem~\ref{th.41}}
In view of Lemma~\ref{le.42} only the second part of the theorem
needs to be proved. In fact, it suffices to prove that conditions
$(A_1)$ and $(A_3)$ imply that $H\in \Re\mathfrak H_2$.

Let therefore conditions $(A_1)$ and $(A_3)$ hold. It follows then
from~\eref{eq.43}, \eref{eq.239}, and Lemma~\ref{le.43} that
\begin{equation}\label{eq.47}
    \mathcal I +\mathscr H_e>0,\qquad
    \mathcal I +\mathscr H_o >0.
\end{equation}
Let us show that~\eref{eq.47} implies positivity of the operator
$\mathcal I +\mathscr H$. Let $L_{2,e}$ and
$L_{2,o}$ be subspaces of $L_2$ consisting of functions that are respectively even and odd with respect to $1\over2$, i.e.,
\[\fl
    L_{2,e}:=\{f\in L_2\,|\,f(1-x)=f(x) \}, \qquad
    L_{2,o}:=\{f\in L_2 \,|\,f(1-x)=-f(x)\}.
\]
$L_{2,e}$ and $L_{2,o}$ are invariant subspaces of the operator~$\mathscr H$. Consider unitary operators $\mathscr A_o: L_2\to L_{2,o}$ and
$\mathscr A_e: L_2\to L_{2,e}$ acting via the formulas
\[\fl
    (\mathscr A_og)(x)=\mathrm{sgn\,}(2x-1)g(|2x-1|),\qquad
    (\mathscr A_eg)(x)=g(|2x-1|), \quad x\in (0,1).
\]
Simple calculations show that for an arbitrary $g\in L_2$ we have
\[ \fl
    ((\mathcal I +\mathscr H)\mathscr A_o g|\mathscr A_o g)
        =((\mathcal I +\mathscr H_o) g|g),
    \qquad
    ((\mathcal I +\mathscr H)\mathscr A_e g|\mathscr A_e g)
        =((\mathcal I +\mathscr H_e)g|g).
\]
Taking into account~\eref{eq.47} and the equality $L_{2,o}\oplus
L_{2,e}=L_2$, we obtain that $\mathcal I +\mathscr H>0$. Therefore
due to~\eref{eq.31} the function~$H$ belongs to $\Re\mathfrak
H_2$.
\end{proofof}

\subsection{Proof of sufficiency in Theorems~\ref{th.14} and~\ref{th.15}}\label{subsec.42}

In Section~\ref{sec.2} the necessity parts of Theorems~\ref{th.14}
and~\ref{th.15} was justified, and we still have to establish the
sufficiency parts. By virtue of Lemma~\ref{le.43} it suffices to
prove the following theorem.

\begin{theorem}\label{th.44}
Let for a sequence $\mathfrak
a=((\lambda_j,\alpha_j))_{j\in\mathbb Z_+}$ conditions
$(A_1)$--$(A_3)$ hold and let $\nu=\nu^\mathfrak a$
and~$\tau=\Theta(H_\nu)$. Then $\tau\in\Re\mathbf L_2$ and
$\mathfrak a=\mathfrak a_\tau$.
\end{theorem}

First we prove the following lemma.

\begin{lemma}\label{le.44}
Let the assumptions of Theorem~\ref{th.44} hold. Then the formulas
\begin{equation} \label{eq.48}
   \widehat P_j=\Psi_\tau(\lambda_j)\alpha_j\Psi^*_\tau(\lambda_j), \qquad
   \widehat Q_k=\Phi_\tau(\lambda_k)\alpha_k\Phi^*_\tau(\lambda_k)
\end{equation}
determine complete systems  $(\widehat P_j)_{j\in\mathbb Z_+}$ and
$(\widehat Q_j)_{j\in\mathbb N}$ of pairwise orthogonal projectors.
\end{lemma}

\begin{proof}
Taking into account Proposition~\ref{pr.37} and
equalities~\eref{eq.43} and~\eref{eq.242}, we arrive at the
equality $\mathscr U^e_{\mathfrak a,\tau}=\mathcal I =\mathscr
U^o_{\mathfrak a,\tau}$ for $\tau=\Theta(H_{\nu})$ and
$\nu=\nu^\mathfrak a$. Hence, by virtue of~\eref{eq.238} and
Lemma~\ref{le.212}, we conclude that
\begin{equation}\label{eq.410}
    \sum_{n=0}^\infty\sum_{\lambda_j\in\Delta_n}\widehat P_j
        =\mathcal I
    = \sum_{n=1}^\infty\sum_{\lambda_j\in\Delta_n}\widehat Q_j,
\end{equation}
and that both series converge in the strong operator topology. We show
that $(\widehat P_j)_{j\in\mathbb Z_+}$ is the sequence of pairwise orthogonal projectors; the proof for~$(\widehat Q_j)_{j\in\mathbb N}$ is analogous.

In view of Lemma~\ref{le.212},
\begin{equation*}
    \sum_{n=1}^\infty\|P_{n,0}
        -\sum_{\lambda_j\in\Delta_n}\widehat P_j\|^2 <\infty.
\end{equation*}
Thus there exists a natural $N_0$ such that
\begin{equation}\label{eq.411}
    \sum_{n=N_0}^\infty\|P_{n,0}
        -\sum_{\lambda_j\in\Delta_n}\widehat P_j \|^2 <1.
\end{equation}
Moreover, $N_0$ can be taken so large that~$(A_2)$ holds, i.e., that for every~$N\ge N_0$ it holds
\begin{equation}\label{eq.412}
    \sum_{n=0}^{N}\sum\limits_{\lambda_j\in\Delta_n}\rank \alpha_j
        = (N+1)r.
\end{equation}
Let us fix $N> N_0$  and set
\[
    A=\sum_{n=N+1}^\infty \sum_{\lambda_j\in\Delta_n}\widehat P_j,
        \qquad
    P=\sum_{n=0}^{N}P_{n,0},
        \qquad
    A_k=\sum_{\lambda_j\in\Delta_k}\widehat P_j.
\]
Since $\rank \widehat P_j=\rank\alpha_j$ for all $j\in\mathbb Z_+$ in view of~\eref{eq.29} and~\eref{eq.48},  we conclude by virtue of~\eref{eq.412} that, for all~$k > N_0$,
\begin{equation}\label{eq.414}
    \rank A_k\le \sum\limits_{\lambda_j\in\Delta_k} \rank\widehat P_j
        =\sum\limits_{\lambda_j\in\Delta_k}\rank\alpha_j
        = r=\rank P_{k,0}.
\end{equation}
Since $(P_{n,0})_{n\in\mathbb Z_+}$  is a complete sequence of
pairwise orthogonal projectors, by virtue of~\eref{eq.411},
\eref{eq.414}, and Lemma~\ref{le.b2} we obtain that
\begin{equation}\label{eq.415}
     \codim\Ran A\ge\rank P=(N+1)r.
\end{equation}
It follows from~\eref{eq.410}, \eref{eq.412}, and~\eref{eq.415}
that
\[
 \eqalign{%
               A+\sum\limits_{n=0}^N \sum\limits_{\lambda_j\in\Delta_n}\widehat P_j=\mathcal I,\\
               \sum\limits_{n=0}^N\sum\limits_{\lambda_j\in\Delta_n}
                \rank\widehat P_j
               =  \sum\limits_{n=0}^N\sum\limits_{\lambda_j\in\Delta_n}
                    \rank \alpha_j = (N+1)r  \le\codim\Ran A.}
\]
The above relations imply by virtue of Lemma~\ref{le.b3}
that the set $\{\widehat P_j :\lambda_j\le \pi N+\pi/2\}$
is a set of pairwise orthogonal projectors. Since $N$ was
arbitrary, we conclude that~$(\widehat P_j)_{j\in\mathbb Z_+}$ is
a sequence of pairwise orthogonal projectors. The proof is complete.
\end{proof}

\begin{proofof}{Proof of Theorem~\ref{th.44}.}
Let the assumptions of Theorem~\ref{th.44} hold and let $\widehat P_j$
and  $\widehat Q_j$ be operators of Lemma~\ref{le.43}. It
follows from Theorems~\ref{th.41} and~\ref{th.13} that
$\tau\in\Re\mathbf L_2$. Theorem~\ref{th.44} will be proved if we
show that
\begin{equation}\label{eq.417}
    \Ran \widehat P_j\subset \ker(S_\tau-\lambda^2_j\mathcal I), \qquad
        j\in\mathbb Z_+.
\end{equation}
Indeed, \eref{eq.417} together with~\eref{eq.410} imply that
$\lambda_j(\tau)=\lambda_j$ for all $j\in\mathbb Z_+$. The last
equality  and~\eref{eq.417} yield the
inequality~$P_{j,\tau}-\widehat P_j\ge 0$, $j\in\mathbb Z_+$.
However, in view of~\eref{eq.410} and~\eref{eq.215} we have
\[
    \sum_{j=0}^\infty (P_{j,\tau}-\widehat P_j) =0.
\]
Thus $P_{j,\tau}-\widehat P_j=0$, $j\in\mathbb Z_+$,  and, therefore,
(see~\eref{eq.216} and~\eref{eq.48})
\[
    \Psi_\tau(\lambda_j)[\alpha_j(\tau)-\alpha_j]\Psi_\tau^*(\lambda_j)=0,
        \qquad  j\in\mathbb Z_+.
\]
Now by virtue of~\eref{eq.29} we obtain the equality
$\alpha_j(\tau)=\alpha_j$, $j\in\mathbb Z_+$, i.e., the equality
$\mathfrak a=\mathfrak a_\tau$.

It thus indeed only remains to prove~\eref{eq.417}. In view
of~\eref{eq.27} and~\eref{eq.21} it is enough to show that
 $\varphi(1,\lambda_j,\tau)\alpha_j=0$, $j\in\mathbb Z_+$. Let $j,k\in\mathbb Z_+$ and $c,d\in \mathbb C^r$.
Since $\tau=\tau^*$, taking into account~\eref{eq.27}
and~\eref{eq.21} and integrating by parts, we find that
\begin{eqnarray*}
    \lambda_j (\Psi(\lambda_j)c|\Psi(\lambda_k)d)_{\mathbb C^r}
        &= \Bigl(\Bigl(\frac{\rmd}{\rmd x}-\tau\Bigr)
            \Phi(\lambda_j)c|\Psi(\lambda_k)d\Bigr)_{\mathbb C^r}\\
        &= 2\Bigl(\varphi(1,\lambda_j,\tau)c
            |\psi(1,\lambda_k,\tau)d\Bigr)_{\mathbb C^r}
            +\lambda_k \Bigl(\Phi(\lambda_j)c|\Phi(\lambda_k)d\Bigr)_{\mathbb C^r}
\end{eqnarray*}
and, therefore,
$$
   \lambda_j \Psi^*(\lambda_k)\Psi(\lambda_j)-\lambda_k \Phi^*(\lambda_k)\Phi(\lambda_j)=2\psi^*(1,\lambda_k,\tau)\varphi(1,\lambda_j,\tau).
$$
It follows from the orthogonality of $(\widehat P_j)_{j\in\mathbb
Z_+}$ and $(\widehat Q_j)_{j\in\mathbb N}$ and
relations~\eref{eq.29} that
$$
    \alpha_k \Psi_\tau^*(\lambda_k)\Psi_\tau(\lambda_j)\alpha_j=0
    = \alpha_k\Phi_\tau^*(\lambda_k)\Phi_\tau(\lambda_j)\alpha_j,
    \qquad k\ne j.
$$
Using this, we obtain the equality
\begin{equation}\label{eq.420}
    \left(\sum\limits_{\lambda_k\in\Delta_n} (-1)^n\alpha_k
        \psi^*(1,\lambda_k,\tau)\right)\varphi(1,\lambda_j,\tau)\alpha_j=0,
        \qquad n\in \mathbb N, \lambda_j\notin\Delta_n.
\end{equation}
Since (see~\eref{eq.213})
\[
    \psi(1,\lambda_k,\tau)= \cos\lambda_k \,I +\int_0^1 (\cos\lambda_k t)K _{\tau,N}(1,t)\,\rmd t,
    \qquad k\in\mathbb N,
\]
we can use the Riemann--Lebesgue lemma and the asymptotic behavior
of the sequences $(\lambda_k)$ and $(\alpha_k)$ to prove that
\[
    \lim\limits_{n\to\infty} \sum\limits_{\lambda_k\in\Delta_n}(-1)^n\psi(1,\lambda_k,\tau)\alpha_k
     = \lim\limits_{n\to\infty}\sum\limits_{\lambda_k\in\Delta_n}\alpha_k
        =I.
\]
Passing to the limit in~\eref{eq.420} as $n\to \infty$, we derive
the equality $\varphi(1,\lambda_j,\tau)\alpha_j=0$.
Theorem~\ref{th.44} is proved.
\end{proofof}

\subsection{Proof of Proposition~\ref{pr.16} and Theorems~\ref{th.1uniq} and~\ref{th.17}} \label{subsec.43}

\begin{proofof}{Proof of Proposition~\ref{pr.16}.}
It follows from Theorems~\ref{th.14} and \ref{th.15} and
Lemma~\ref{le.43} proved above that it is sufficient to prove the
implication $(\mathfrak a\in\mathfrak A) \Longrightarrow
(\lambda_0=0, \alpha_0>0)$. This fact, however, is an immediate
corollary of Theorem~\ref{th.24}.
\end{proofof}

\begin{proofof}{Proof of Theorem~\ref{th.1uniq}.}
$(i)$ Let $\tau_1,\tau \in \Re \mathbf L_2$ and $\mathfrak
a_{\tau_1}=\mathfrak a_{\tau}=\mathfrak a$. In view
of~\eref{eq.250} $\mathscr U^e_{\frak a,\tau_1}=\mathcal
I=\mathscr U^e_{\frak a,\tau}$. We next prove the following
implication:
\begin{equation}\label{eq.424}
\mathscr U^e_{\frak a,\tau_1}=\mathscr U^e_{\frak
a,\tau}\Longrightarrow\mathscr K_{\tau_1,N}=\mathscr K_{\tau,N}.
\end{equation}
 According to~\eref{eq.43}, $\mathscr U^e_{\frak a,0}=\mathcal I
+\mathscr H_e$, and thus we conclude that the operator $(\mathscr
U^e_{\frak a,0}- \mathcal I)$ belongs to $\mathcal B_2$. Since
$\mathscr U^e_{\frak a,0}>0$ (see~\eref{eq.47}),  the operator
$\mathscr U^e_{\frak a,0}$ admits a factorization:
\[
      \mathscr U^e_{\frak a,0}=(\mathcal I +\mathscr K)^{-1}(\mathcal I+\mathscr K^*)^{-1},\qquad
       \mathscr K\in\mathcal B_2^+.
\]
It follows from the equality $\mathscr U^e_{\frak
a,\tau_1}=\mathscr U^e_{\frak a,\tau}$ and~\eref{eq.242} that
\begin{eqnarray*}
(\mathcal I+\mathscr K_{\tau_1,N})(\mathcal I+\mathscr K)^{-1}
    &(\mathcal I+\mathscr K^*)^{-1}(\mathcal I+\mathscr K^*_{\tau_1,N})\\
                  & =(\mathcal I+\mathscr K_{\tau,N})(\mathcal I+\mathscr K)^{-1}
                     (\mathcal I+\mathscr K^*)^{-1}(\mathcal I+\mathscr K^*_{\tau,N}).
\end{eqnarray*}
Therefore, in view of Theorem~\ref{th.a1} we have
\[(\mathcal I+\mathscr K_{\tau_1,N})(\mathcal I+\mathscr K)^{-1}
 =(\mathcal I+\mathscr K_{\tau,N})(\mathcal I+\mathscr K)^{-1},
\] which implies that $\mathscr K_{\tau_1,N}=\mathscr K_{\tau,N}$.

It follows now from Lemma~\ref{le.22}~\emph{(iii)} and~\eref{eq.424} that~$\tau_1=\tau$.

$(ii)$  Let $q_1,q\in\Re\mathfrak M$ and  $\mathfrak
b_{q_1}=\mathfrak b_q$. Let us fix $\tau_1\in b^{-1}(q_1)$ and
$\tau\in b^{-1}(q)$. Since  $\mathfrak b_{q_1}=\mathfrak b_q$, the
sequences $\mathfrak a=\mathfrak a_\tau$ and $\mathfrak
a_1=\mathfrak a_{\tau_1}$ differ at most in the first element. Thus, taking into account~\eref{eq.250}, we
obtain that
\begin{equation}\label{eq.455}
        \mathscr U^o_{\mathfrak a,\tau}=\mathscr U^o_{\mathfrak a_\tau,\tau}=\mathcal I
        = \mathscr U^o_{\mathfrak a_{\tau_1},\tau_1}=\mathscr U^o_{\mathfrak a,\tau_1}.
\end{equation}
Reasoning as in the proof of~\eref{eq.424}, we establish the
implication
\[
    (\mathscr U^o_{\frak a,\tau_1} = \mathscr U^o_{\frak a,\tau})
            \Longrightarrow
    (\mathscr K_{\tau_1,D}=\mathscr K_{\tau,D}).
\]
Therefore, taking into account~\eref{eq.455} and the statement $(ii)$ of Lemma~\ref{le.22}, we get that $b(\tau_1)=b(\tau)$, i.e., $q_1=q$.
The proof is complete.
\end{proofof}

\begin{proofof}{Proof of Theorem~\ref{th.17}.}
 Part $(i)$ is a straightforward consequence of Theorems~\ref{th.44} and~\ref{th.1uniq}. Let us prove part $(ii)$. Assume that $q\in\Re\mathfrak M$
 and that $\mathfrak b=\mathfrak b_q$ belongs to~$\mathfrak B$.
 Augmenting the sequence $\mathfrak b$ to a sequence $\mathfrak a$ with $\lambda_0=0$ and $\alpha_0=I$, we get by Proposition~\ref{pr.16} that~$\mathfrak a\in  \mathfrak A$. Set
 \[
     \nu:=\nu^{\mathfrak a}=I\delta_0 +\mu^{\mathfrak b},\qquad \tau:=\Theta(H_\nu).
 \]
 From the already proved statement $(i)$ we obtain that $\mathfrak a=\mathfrak a_{\tau}$. Then $\mathfrak b=\mathfrak b_{\tilde q}$ for~$\tilde q=b(\tau)$. It follows from
Theorem~\ref{th.1uniq} that  $q=\tilde q=b(\tau)$. The proof is complete.
\end{proofof}

\ack  {The authors express their gratitude
to DFG for financial support of the project \mbox{436 UKR 113/84} and
thank the Internationales Zentrum f\"{u}r Komplexe Sys\-te\-me (IZKS)
of Bonn University for the warm hospitality. The research of the first author was partially supported by the Ukrainian DFFD grant F28.1/017.}

\appendix

\def\thesection{Appen\-dix~\Alph{section}}

\section{\textbf{The spaces}}\label{add.1}

\def\thesection{\Alph{section}}

\subsection{Sobolev spaces}

Suppose that $X$ is a Banach space. Denote by $L_p((a,b),X)$,
$1\le p<\infty$, the Banach space of all strongly measurable functions
$f:(a,b)\to X$, for which the norm
\[
    ||f||_{L_p((a,b),X)}
        :=\left(\int_a^b\|f(t)\|_{X}^p\,\rmd t\right)^{1/p}
\]
is finite. Denote by~$C^k([a,b],X)$ the Banach space of $k$ times
continuously differentiable functions~$[a,b]\to X$ with the standard
norm.

We also denote by~$W_p^{1}((0,1),X)$, $1\le p<\infty$, the Sobolev
space that is the completion of the linear space $C^1([0,1],X)$ by
the norm
\begin{equation*}
    \|f\|_{W_p^1}
    := {\left(\int_0^1(\|f(t)\|^p_X\,\rmd t\right)}^{1/p}
        +{\left(\int_0^1\|f'(t)\|^p_X)\,\rmd t\right)}^{1/p}.
\end{equation*}
Every function $f\in W_p^1((0,1),X)$ has the derivative $f'$ belonging to
$L_p((0,1),X)$.

We denote by~$M_r$ the Banach algebra of $r\times r$
matrices with complex entries. In the standard way the algebra
$M_r$ is identified with the Banach algebra of linear operators
$A:\mathbb C^r\to\mathbb C^r$ and inherits the operator norm
\[
      \|A\|=\sup\limits_{\|y\|_{\mathbb{C}^r}=1}\|Ay\|_{\mathbb{C}^r},
\]
where $\|\,\cdot\,\|_{\mathbb{C}^r}$ is the Euclidean norm
of~$\mathbb{C}^r$  generated by the standard scalar
product~$(\cdot\,|\,\cdot)_{\mathbb{C}^r}$. We shall write $L_2$
instead of $L_2((0,1),\mathbb{C}^r)$  and denote by
$(\cdot\,|\,\cdot)$ the scalar product in $L_2$, i.e.,
\[
    (f\,|\,g)=\int_0^1 (f(x)\,|\,g(x))_{\mathbb C^r}\,\rmd x,  \qquad f,g\in L_2.
\]
We also use the following notations:
\[
    \mathbf L_p:=L_p((0,1),M_r), \qquad
    \Re \mathbf L_p:=\{u\in \mathbf L_p \mid u=u^*\}, \qquad
    p\geq1,
\]
\[
    W_p^s:=W_p^s((0,1),\mathbb{C}^r),\qquad
    \mathbf W_p^s:=W_p^s((0,1),M_r) \qquad(s\in\mathbb{Z}, p\geq 1).
\]
Set $W_{2,0}^1(0,1)=\{f\in W_2^{1}(0,1)\,|\,f(0)=f(1)=0\}$. Recall
(see~\cite{AF}) that $W^{-1}_2(0,1)$ is the dual space
of~$W^{1}_{2,0}(0,1)$. Denote by $\mathbf
W^{-1}_2$ the Banach space of matrices
$f=(f_{ij})_{1\le i,j\le r}$ with entries $f_{ij}\in
W^{-1}_2(0,1)$ and with the inherited norm
\[
    \|f\|_{\mathbf W^{-1}_2}
        =\left(\sum_{1\le i,j\le r}\|f_{ij}\|^2_{W^{-1}_2(0,1)}\right)^{1/2}.
\]
Let $\Re\mathbf W^{-1}_2:=\{f\in\mathbf W^{-1}_2 \mid f=f^*\}$.
Here $f^*:=(\overline f_{ji})_{1\le i,j\le r}$ for
$f=(f_{ij})_{1\le i,j\le r}$.

\subsection{Factorization of operators}

We state below some necessary facts related to the factorization
of Fredholm operators~\cite[Ch.4]{GGK1}. Denote by $\mathcal B_2$
the Hilbert space of all Hilbert--Schmidt operators in~$L_2$.
Every operator $\mathscr K\in\mathcal B_2$ is an integral operator
with kernel $K\in L_2((0,1)^2,M_r)$. Set $\Omega:=\{(x,t) | 0\le
t\le x\le 1\}, \Omega^-:=[0,1]^2\setminus\Omega$ and denote by
$\mathcal B^+_2$ and $\mathcal B^-_2$ the subalgebras of $\mathcal
B_2$ consisting of all operators with kernels $K$ satisfying the
condition $K(x,t)=0$ a.e. in $\Omega^-$ and $K(x,t)=0$ a.e. in
$\Omega$, respectively. It is obvious that $\mathcal B_2=\mathcal
B^+_2\oplus\mathcal B^-_2$; moreover, the operators in $\mathcal
B_2^\pm$ are Volterra ones~\cite[Ch.~IV]{GGK1}. Denote by $\mathscr P^+$
an orthogonal projector in~$\mathcal B_2$ onto~$\mathcal B^+_2$.

We say that an operator $\mathcal I +\mathscr K$ with $\mathscr
K\in\mathcal B_2$ admits a factorization in $\mathcal B_2$ if
\begin{equation}\label{eq.a1}
   \mathcal I +\mathscr K=(\mathcal I +\mathscr K_+)^{-1}(\mathcal I +\mathscr K_-)^{-1}
\end{equation}
with some $\mathscr K_+\in\mathcal B_2^+$ and $\mathscr
K_-\in\mathcal B_2^-$.

\begin{theorem}\label{th.a1}
Assume that $\mathscr K\in\mathcal B_2$ with kernel $K$ is such that
$\mathcal I+\mathscr K$ admits a factorization in $\mathcal B_2$.
  Then the operators $\mathscr K_+$ and $\mathscr K_-$
in~\eref{eq.a1} are determined uniquely by $\mathscr K$. Moreover,
the operators $\mathscr K_+$ and $\mathscr K^*_-$ are solutions for $\mathscr X\in\mathcal B_2$  of the equations
\begin{equation}\label{eq.a2}
 \mathscr X +\mathscr P^+\mathscr K +\mathscr P^+\mathscr X\mathscr K
=0,\qquad \mathscr X +\mathscr P^+\mathscr K^*+\mathscr
P^+\mathscr X\mathscr K^* =0
\end{equation}
respectively.
For the operator $\mathcal I+\mathscr K$ to admit a factorization
in $\mathcal B_2$, it is necessary and sufficient that one of the
following conditions should hold:
  \begin{enumerate}
    \item for every $a\in [0,1]$ the integral equation
\[
    f(x)+\int_0^a K(x,t) f(t)\,\rmd t =0, \qquad  x\in (0,1),
\]
has only the trivial solution in $L_2$;
  \item at least one of the equations in~\eref{eq.a2} has a solution.
  \end{enumerate}
\end{theorem}

In terms of the kernels the first equation in~\eref{eq.a2}
takes the form
\begin{equation}\label{eq.a3}
    X(s,t)+K(s,t) +\int_0^s X(s,\xi) K(\xi,t) \,\rmd \xi =0,
        \qquad (s,t)\in\Omega.
\end{equation}

Theorem~\ref{th.a1} yields the following connection between the Krein accelerants
and the Krein equation, namely:

\begin{proposition}\label{th.a1a}
Assume that $H\in L_2((-1,1),M_r)$ and that~$\mathscr H$ is an integral operator of~\eref{eq.30}.
Then the following statements are equivalent:
 \begin{enumerate}
    \item the function $H$ belongs to the class $\mathfrak H_2$;
    \item the operator $\mathcal I+\mathscr H$ admits a factorization;
    \item the Krein equation~\eref{eq.19} has a solution in $L_2(\Omega,M_r)$.
 \end{enumerate}
\end{proposition}

Theorem~\ref{th.a1} also implies the following lemma.
\begin{lemma}\label{le.a2}
If $K$ belongs to~$L_2((0,1)^2,M_r)$, then equation~\eref{eq.a3}
has at most one solution, and the equation
\[
    X(s,t) +\int_0^s X(s,\xi) K(\xi,t) \,\rmd t =0,
        \qquad (s,t)\in\Omega,
\]
has only zero solution in $L_2(\Omega,M_r)$.
\end{lemma}

We denote by~$G_2$ the set of all functions
 $K:[0,1]^2\rightarrow M_r$
having the property that the mappings
\[
    [0,1]\ni x \mapsto K(x,\cdot)\in \mathbf L_2,\qquad
    [0,1]\ni t \mapsto K(\cdot,t)\in \mathbf L_2
\]
are continuous on the interval $[0,1]$. It is obvious that
$G_2\subset L_2((0,1)^2,M_r)$. The set $G_2$ becomes a Banach
space upon introducing the norm
\[
    \|K\|_{G_2} :=
        \max\Bigl\{\max_{x\in[0,1]}\|K(x,\cdot)\|_{\mathbf L_2},
        \max_{t\in[0,1]}\|K(\cdot,t)\|_{\mathbf L_2}\Bigr\}.
\]
Denote by $G^+_2$ and $G^-_2$ the subspaces of~$ G_2$ consisting
of all functions~$K$ that satisfy the condition $K(x,t)=0$ a.e. in
$\Omega^-$ and $K(x,t)=0$ a.e. in $\Omega$, respectively.

Also denote by $\mathfrak G_2$ a Banach algebra consisting of integral
operators~$\mathscr{K}$ with kernels~$K$ belonging to~$G_2.$ The
norm in  $\mathfrak G_2$ is defined via
\[
    \|\mathscr K\|_{\mathfrak G_2}=\|K\|_{G_2},\qquad
    \mathscr{K}\in\mathfrak G_2.
\]
Denote by $\mathfrak G^+_2$ and $\mathfrak G^-_2$ the set of those
$\mathscr K\in\mathfrak G_2$, whose kernels belong respectively to~$G_2^+$ and $G_2^-$. Observe that~$\mathfrak G^+_2$  and~$\mathfrak G^-_2$
form closed subalgebras of~$\mathfrak G_2$ and that~$\mathfrak
G_2^\pm\subset \mathcal B_2^\pm$.

\begin{lemma}\label{le.a3}
The mapping $A\mapsto (\mathcal I+A)^{-1}-\mathcal I=:\eta(A)$ acts continuously in
$G_2^+$.
\end{lemma}

\begin{proof}
Using the Cauchy--Bunyakowski inequality, we find that for every
$A\in G_2^+$
\begin{equation}\label{eq.a4}
    ||A^{n+1}||_{G_2^+}\le(n!)^{-1/2}||A||_{G_2^+}^{n+1},\qquad
               n\in\mathbb Z_+.
\end{equation}
Let $A,B\in G_2^+$ and $||A||_{G_2^+}, ||B||_{G_2^+}\le r$ for
some fixed~$r>0$. Since
\[
    \eta(A)-\eta(B)=\sum_{n=1}^{\infty}[(-A)^n -(-B)^n],
\]
in view of~\eref{eq.a4} and the identity~$A^{n+1}-B^{n+1}=\sum_{k=0}^{n}A^{n-k}(A-B)B^{k}$, straightforward
calculations show that the mapping $\eta$ acts continuously in~$G^+_2$ and that
\[
    \|\eta(A)-\eta(B)\|_{G_2^+}\le C^2(r)\|A-B\|_{G_2^+}
\]
with
\[
    C(r):=1+\sum_{k=0}^{\infty}(k!)^{-1/2}r^{k+1}.
\]
The proof is complete.
\end{proof}


\def\thesection{Appen\-dix~\Alph{section}}
\section{ \textbf{Lemmas about orthogonal projectors}}\label{add.2}
\def\thesection{\Alph{section}}

Suppose that~$H$ is a Hilbert space and denote by~$\mathcal B(H)$ the
Banach algebra of all bounded linear operators in~$H$.

\begin{lemma}\label{le.b1}
Let $(P_n)_{n=1}^{\infty}$ and $(G_n)_{n=1}^{\infty}$ be
sequences of pairwise orthogonal projectors in $H$ of finite rank such that
$
\sum_{n=1}^{\infty}P_n=\sum_{n=1}^{\infty}G_n=I_H$ and
$\sum_{n=1}^{\infty}||P_n-G_n||^2<\infty$. Then
there exists $N_0\in\mathbb N$ such that, for all $N\ge N_0$,
\[
    \sum_{n=1}^N\rank P_n=\sum_{n=1}^N\rank G_n.
\]
\end{lemma}

\begin{proof}
Let us choose a natural $N_0$ such that
$\sum_{n=N_0}^\infty\|P_n- G_n\|^2 <\frac14$. Fix an arbitrary
natural $N\ge N_0$ and consider the orthogonal projectors
$P=\sum_{n=1}^N P_n$ and $G=\sum_{n=1}^N G_n$. We shall now show that~$\|P-G\| <1$. Indeed, the assumptions of the lemma imply that
\[
    P-G=\sum_{n=N+1}^\infty (P_n- G_n)=\sum_{n=N+1}^\infty P_n(P_n- G_n)
        +\sum_{n=N+1}^\infty (P_n- G_n)G_n.
\]
Thus for every~$f\in H$ of norm~$1$ we get the inequality
\[
    |((P-G)f|f)_H|\le \sum_{n=N+1}^\infty(\|P_nf\|+\|G_nf\|)\|(P_n- G_n)f\|.
\]
Using the Cauchy--Bunyakowski inequality, we obtain that
\[
    \|P-G\|^2\le 4\sum_{n=N+1}^\infty\|P_n-G_n\|^2 <1.
\]
The inequality $\|P-G\|<1$ implies (see~\cite[Ch.1]{Kat})
that $\rank P=\rank G$, and the proof is complete.
\end{proof}

\begin{lemma}\label{le.b2}
Assume that $(A_j)_{j=1}^\infty$ is a sequence in~$\mathcal B(H)$ and that $(G_j)_{j=1}^\infty$ is a sequence of
pairwise orthogonal projectors such that the following holds:
\begin{enumerate}
   \item   the series $\sum_{j=1}^\infty A_j$  converges in the strong operator topology to an operator $A$;
   \item   the orthogonal projector $G:=I_H -\sum_{j=1}^\infty G_j$ is of finite rank;
   \item   $\sum_{j=1}^\infty \|A_j-G_j\|^2 < 1$ and $\rank A_j\le\rank G_j<\infty$ for every $j\in\mathbb N$.
\end{enumerate}
Then $\codim \Ran A \ge \rank G$.
\end{lemma}

\begin{proof} It follows from the assumptions of the lemma that, for every
$f\in L_2$,
\[
    \sum\limits_{j=1}^\infty \|(A_jG_j -G_j)f\|\le
    \sum\limits_{j=1}^\infty \|A_j -G_j\|\|G_jf\|\le
    \biggl(\sum\limits_{j=1}^\infty \|A_j-G_j\|^2  \biggr)^{1/2} \|f\|
\]
and, therefore, the series $\sum_{j=1}^\infty (A_jG_j -G_j)$
strongly converges to some operator $B\in\mathcal B(H)$ with
$\|B\|<1$. Hence the series $\sum_{j=1}^\infty A_jG_j$ strongly
converges to the operator $\tilde A:=(I_H+B)(I_H -G)$. Since the
operator $I_H +B$ is invertible, $\tilde A$ has a closed range and
$\codim\Ran\tilde A=\rank G$. Since $\rank A_j\le\rank G_j<\infty$, we have
\begin{equation}\label{eq.2b}
   \Ran A_jG_j=\Ran A_j
\end{equation}
for all $j\in\mathbb N$. Indeed, if~\eref{eq.2b} does not hold for some~$k\in\mathbb N$, then there exists $c\in\Ran G_k$ such that~$A_k c=0$. Therefore
$\|(A_k -G_k)c\|=\|c\|$, which implies that~$\|A_k -G_k\|\ge 1$ thus contradicting assumption~$(iii)$.

We denote by $X$ the closed linear hull of the
set~$\cup_{j\in\mathbb N} \Ran A_j$.  It is obvious that $\Ran
A\subset X$. It follows from~\eref{eq.2b} and the definition of
the operator $\tilde A$ that $\Ran\tilde A$ contains the set
$\cup_{j\in\mathbb N} \Ran A_j$. But, as already noted,
$\Ran\tilde A$ is a closed set. Thus $ \Ran A\subset
X\subset\Ran\tilde A$ and, therefore, $\codim\Ran A
\ge\codim\Ran\tilde A=\rank G$. The proof is complete.
\end{proof}

\begin{lemma}\label{le.b3}
Let $\{A_j\}_{j=0}^n$ be a set of self-adjoint operators
from the algebra $\mathcal B(H)$ that are of finite rank for $j\ne0$. If
\[
    \sum\limits_{j=0}^n A_j=I_H, \qquad
    \sum\limits_{j=1}^n \rank A_j\le \codim\Ran A_0,
\]
then $\{A_j\}_{j=0}^n$ is the set of pairwise orthogonal
projectors.
\end{lemma}

\begin{proof} It follows from the assumptions of the lemma that the space~$H$ is the direct sum of the subspaces $\Ran A_j$ for $j=0,\dots, n$ and that~$A_k(A_k-I_H) + \sum_{j\ne k} A_jA_k=0$ for every~$k\le n$.
It is obvious that the equality
is possible only if all the summands on the left hand side are equal to
zero. This immediately yields the statement of the lemma.
\end{proof}

\end{document}